\documentclass{amsart}

\usepackage{mathrsfs}
\usepackage{amsthm}
\usepackage{amssymb}
\usepackage{mathtools}
\usepackage{breqn}
\usepackage{hyperref}
\usepackage{enumerate}
\usepackage{caption}
\usepackage{xifthen}

\hypersetup{
    colorlinks,
    linkcolor={red},
    citecolor={green},
    urlcolor={blue}
}

\usepackage[backref=true,backend=biber]{biblatex}

\numberwithin{equation}{section}

\theoremstyle{plain}
\newtheorem{theorem}{Theorem}
\newtheorem{lemma}{Lemma}
\newtheorem{corollary}{Corollary}
\newtheorem{proposition}{Proposition}
\theoremstyle{definition}
\newtheorem{definition}{Definition}

\providecommand{\set}[2][]{
	\ifthenelse{\isempty{#1}}{
		\left\{#2\right\}
	}{
		\left\{\,#1\;\middle|\;#2\,\right\}}
	}
\DeclareMathOperator{\pow}{Sb}

\providecommand{\N}{\mathbb{N}}
\DeclareMathOperator{\lcm}{lcm}
\DeclareMathOperator{\perm}{Perm}
\DeclareMathOperator{\permg}{\mathbf{Perm}}
\DeclareMathOperator{\id}{id}
\DeclareMathOperator{\orb}{Orb}
\providecommand{\Z}{\mathbb{Z}}
\DeclareMathOperator{\autg}{\mathbf{Aut}}
\DeclareMathOperator{\aut}{Aut}
\DeclareMathOperator{\con}{Con}
\DeclareMathOperator{\cg}{Cg}
\DeclareMathOperator{\sg}{Sg}
\DeclareMathOperator{\conl}{\mathbf{Con}}
\providecommand{\F}{\mathbb{F}}

\DeclareMathOperator{\rps}{RPS}
\DeclareMathOperator{\prps}{PRPS}
\DeclareMathOperator{\tour}{Tour}
\DeclareMathOperator{\chir}{Chir}
\DeclareMathOperator{\sgn}{Sgn}
\DeclareMathOperator{\obv}{Obv}

\begin{document}
\title{Multiplayer Rock-Paper-Scissors}
\author[C. Aten]{Charlotte Aten}
\address{Mathematics Department\\
University of Rochester\\Rochester 14627\\USA}
\urladdr{\href{http://aten.cool}{aten.cool}}
\email{\href{mailto:charlotte.aten@rochester.edu}{charlotte.aten@rochester.edu}}
\thanks{Thanks to Jonathan Pakianathan and Clifford Bergman for their helpful comments. Thanks to Scott Kirila for pointing out the result of Joris, Oestreicher, and Steinig we use in \autoref{sec:rps_magmas}. A short version of this paper appeared in the proceedings of the 2018 Algebras and Lattices in Hawai'i conference\cite{atenrps}. This research was supported in part by the people of the Yosemite Valley.}


\begin{abstract}
We study a class of algebras we regard as generalized Rock-Paper-Scissors games. We determine when such algebras can exist, show that these algebras generate the varieties generated by hypertournament algebras, count these algebras, study their automorphisms, and determine their congruence lattices. We produce a family of finite simple algebras.
\end{abstract}

\maketitle

\tableofcontents

\section{Introduction}
The game of Rock-Paper-Scissors (RPS) involves two players simultaneously choosing either rock (\(r\)), paper (\(p\)), or scissors (\(s\)). Informally, the rules of the game are that ``rock beats scissors, paper beats rock, and scissors beats paper''. That is, if one player selects rock and the other selects paper then the latter player wins, and so on. If two players choose the same item then the round is a tie.

A \emph{magma} is an algebra \(\mathbf{A}\coloneqq(A,f)\) consisting of a set \(A\) and a single binary operation \(f\colon A^2\to A\). We will view the game of RPS as a magma. We let \(A\coloneqq\{r,p,s\}\) and define a binary operation \(f\colon A^2\to A\) where \(f(x,y)\) is the winning item among \(\{x,y\}\). This operation is given by the table in \autoref{fig:rps_table} and completely describes the rules of RPS. In order to play the first player selects a member of \(A\), say \(x\), at the same time that the second player selects a member of \(A\), say \(y\). Each player who selected \(f(x,y)\) is the winner. Note that it is possible for both players to win, in which case we have a tie.

\begin{figure}[ht]
	\begin{tabular}{r|*{3}{c}}
		& \(r\) & \(p\) & \(s\) \\ \hline
		\(r\) & \(r\) & \(p\) & \(r\) \\
		\(p\) & \(p\) & \(p\) & \(s\) \\
		\(s\) & \(r\) & \(s\) & \(s\)
	\end{tabular}
	\caption{The \(\rps\) operation}
	\label{fig:rps_table}
\end{figure}

In general we have a class of \emph{selection games}, which are games consisting of a collection of items \(A\), from which a fixed number of players \(n\) each choose one, resulting in a tuple \(a\in A^n\), following which the round's winners are those who chose \(f(a)\) for some fixed rule \(f\colon A^n\to A\). We refer to an algebra \(\mathbf{A}\coloneqq(A,f)\) with a single basic \(n\)-ary operation \(f\colon A^n\to A\) as an \emph{\(n\)-ary magma} or an \emph{\(n\)-magma}. We will sometimes abuse this terminology and refer to an \(n\)-ary magma \(\mathbf{A}\) simply as a \emph{magma}. Each such game can be viewed as an \(n\)-ary magma and each \(n\)-ary magma can be viewed as a game in the same manner, providing we allow for games where we keep track of who is ``player 1'', who is ``player 2'', etc. Again note that any subset of the collection of players might win a given round, so there can be multiple player ties.

The classic RPS game has several desirable properties. Namely, RPS is, in terms we proceed to define,
	\begin{enumerate}
		\item conservative,
		\item essentially polyadic,
		\item strongly fair, and
		\item nondegenerate.
	\end{enumerate}
Let \(\mathbf{A}\coloneqq(A,f)\) be an \(n\)-magma. We say that an operation \(f\colon A^n\to A\) is \emph{conservative} when for any \(a_1,\dots,a_n\in A\) we have that \(f(a_1,\dots,a_n)\in\{a_1,\dots,a_n\}\)\cite[p.94]{bergman}. Similarly we call \(\mathbf{A}\) \emph{conservative} when \(f\) is conservative. We say that an operation \(f\colon A^n\to A\) is \emph{essentially polyadic} when there exists some \(g\colon\pow_{\le n}(A)\to A\) where \(\pow_{\le n}(A)\coloneqq\{\,B\subset A\mid1\le\lvert B\rvert\le n\,\}\) such that for any \(a_1,\dots,a_n\in A\) we have \(f(a_1,\dots,a_n)=g(\{a_1,\dots,a_n\})\). Similarly we call \(\mathbf{A}\) \emph{essentially polyadic} when \(f\) is essentially polyadic. We say that \(f\) is \emph{fair} when for all \(a,b\in A\) we have \(\lvert f^{-1}(a)\rvert=\lvert f^{-1}(b)\rvert\). Let \(A_k\) denote the members of \(A^n\) which have exactly \(k\) distinct components for some \(k\in\N\). We say that \(f\) is \emph{strongly fair} when for all \(a,b\in A\) and all \(k\in\N\) we have \(\lvert f^{-1}(a)\cap A_k\rvert=\lvert f^{-1}(b)\cap A_k\rvert\). Similarly we call \(\mathbf{A}\) \emph{(strongly) fair} when \(f\) is (strongly) fair. Note that if \(f\) (respectively, \(\mathbf{A}\)) is strongly fair then \(f\) (respectively, \(\mathbf{A}\)) is fair, but the reverse implication does not hold. We say that \(f\) is \emph{nondegenerate} when \(\lvert A\rvert>n\). Similarly we call \(\mathbf{A}\) \emph{nondegenerate} when \(f\) is nondegenerate.

Thinking in terms of selection games we say that \(\mathbf{A}\) is conservative when each round has at least one winning player. We say that \(\mathbf{A}\) is essentially polyadic when a round's winning item is determined solely by which items were played, not taking into account which player played which item or how many players chose a particular item. We say that \(\mathbf{A}\) is fair when each item has the same probability of being the winning item (or tying). We say that \(\mathbf{A}\) is strongly fair when each item has the same chance of being the winning item when exactly \(k\) distinct items are chosen for any \(k\in\N\). Note that this is not the same as saying that each player has the same chance of choosing the winning item (respectively, when exactly \(k\) distinct items are chosen). When \(\mathbf{A}\) is \emph{degenerate} (i.e. not nondegenerate) we have that \(\lvert A\rvert\le n\). In this case we have that all members of \(A_{\lvert A\rvert}\) have the same set of components. If \(\mathbf{A}\) is essentially polyadic with \(\lvert A\rvert\le n\) it is impossible for \(\mathbf{A}\) to be strongly fair unless \(\lvert A\rvert=1\).

Extensions of RPS which allow players to choose from more than the three eponymous items are attested historically. The French variant of RPS gives a pair of players \(4\) items to choose among\cite[p.140]{umbhauer}. In addition to the usual rock, paper, and scissors there is also the well (\(w\)). The well beats rock and scissors but loses to paper. The corresponding Cayley table is given in \autoref{fig:french_rps}. This game is not fair, as \(\lvert f^{-1}(r)\rvert=3\) yet \(\lvert f^{-1}(p)\rvert=5\). It is nondegenerate since there are \(4\) items for \(2\) players to chose among. It is also conservative and essentially polyadic.

\begin{figure}[ht]
	\begin{tabular}{r|*{4}{c}}
		& \(r\) & \(p\) & \(s\) & \(w\) \\ \hline
		\(r\) & \(r\) & \(p\) & \(r\) & \(w\) \\
		\(p\) & \(p\) & \(p\) & \(s\) & \(p\) \\
		\(s\) & \(r\) & \(s\) & \(s\) & \(w\) \\
		\(w\) & \(w\) & \(p\) & \(w\) & \(w\)
	\end{tabular}
	\caption{The French variant of \(\rps\)}
	\label{fig:french_rps}
\end{figure}

There has been some recent recreational interest in RPS variants with larger numbers of items from which two players may choose. For example, the game Rock-Paper-Scissors-Spock-Lizard\cite{rpssl} (RPSSL) is attested in the popular culture. The Cayley table for this game is given in \autoref{fig:rpssl}, with \(v\) representing Spock and \(l\) representing lizard. This game is conservative, essentially polyadic, strongly fair, and nondegenerate.

\begin{figure}[ht]
	\begin{tabular}{r|*{5}{c}}
		& \(r\) & \(p\) & \(s\) & \(v\) & \(l\) \\ \hline
		\(r\) & \(r\) & \(p\) & \(r\) & \(v\) & \(r\) \\
		\(p\) & \(p\) & \(p\) & \(s\) & \(p\) & \(l\) \\
		\(s\) & \(r\) & \(s\) & \(s\) & \(v\) & \(s\) \\
		\(v\) & \(v\) & \(p\) & \(v\) & \(v\) & \(l\) \\
		\(l\) & \(r\) & \(l\) & \(s\) & \(l\) & \(l\)
	\end{tabular}
	\caption{The RPSSL operation}
	\label{fig:rpssl}
\end{figure}

Variants of RPS with larger numbers of items appear in the literature as balanced tournaments\cite{chamberland}. Under this combinatorial definition it is well-established that only variants with an odd number of items may exist when the quantity of items to choose from is finite. We detail the connection between our generalization of RPS and tournaments in \autoref{sec:tournaments}. In our language we have an analogous odd-order result.

\begin{proposition}
Let \(\mathbf{A}\) be a selection game with \(n=2\) which is essentially polyadic, strongly fair, and nondegenerate and let \(m\coloneqq\lvert A\rvert\in\N\). We have that \(m\neq1\) is odd. Conversely, for each odd \(m\neq1\) there exists such a selection game.
\end{proposition}

\begin{proof}
Since \(\mathbf{A}\) is nondegenerate and \(n=2\) we must have that \(m>n=2\) and hence \(m\neq1\). Games with \(m=1\) have only one item to choose from and all players always tie.

Since \(\mathbf{A}\) is strongly fair we must have that \(\lvert f^{-1}(a)\cap A_2\rvert=\lvert f^{-1}(b)\cap A_2\rvert\) for all \(a,b\in A\). As the \(m\) distinct sets \(f^{-1}(a)\cap A_2\) for \(a\in A\) partition \(A_2\) and are all the same size we require that \(m\mid\lvert A_2\rvert\). Moreover, as we take \(\mathbf{A}\) to be essentially polyadic we have that \(f(a,b)=f(b,a)\) for all \(a,b\in A\). This implies that each of the \(m\) items must be the winner among the same number of unordered pairs of distinct elements \(\{a,b\}\). Let \(\binom{A}{2}\) denote the collection of unordered pairs of distinct elements in \(A\). That is, we write \(\binom{A}{2}\) to indicate the collection of \emph{\(2\)-sets} of \(A\). We have that \(\lvert\binom{A}{2}\rvert=\binom{m}{2}\) so we require that \(m\mid\lvert\binom{A}{2}\rvert=\binom{m}{2}\), which implies that \(m\) is odd.

It remains to show that such games \(\mathbf{A}\) exist when \(m\neq1\) is odd. Let \(\binom{A}{1}\) denote the collection of singletons of elements in \(A\). That is, we write \(\binom{A}{1}\) to indicate the collection of \emph{\(1\)-sets} of \(A\). Since \(\lvert\binom{A}{1}\rvert=\lvert A\rvert=m\) we can partition \(\binom{A}{1}\) into \(m\) subcollections \(C_1\coloneqq\{C_{1,r}\}_{r\in A}\) indexed on the \(m\) elements of \(A\), each with \(\lvert C_{1,r}\rvert=1\). Since we assume that \(m\neq1\) is odd we have that \(m=2s+1\) for some \(s\in\N\). This implies that
	\[
		\binom{m}{2}=\binom{2s+1}{2}=(2s+1)s=ms
	\]
so we can partition \(\binom{A}{2}\) into \(m\) subcollections \(C_2\coloneqq\{C_{2,r}\}_{r\in A}\) indexed on the \(m\) elements of \(A\), each with \(\lvert C_{2,r}\rvert=s\). With respect to these partitions \(C\coloneqq\{C_1,C_2\}\) we define a binary operation \(f\colon A^2\to A\) by \(f(a,b)\coloneqq r\) when \(\{a,b\}\in C_{k,r}\) for some \(k\in\{1,2\}\). This map is well-defined since each \(\{a,b\}\) contains either \(1\) or \(2\) distinct elements and thus belongs to a unique member of one of the partitions \(C_k\). In order to see that the resulting magma \(\mathbf{A}\coloneqq(A,f)\) is essentially polyadic let \(g\colon\pow_{\le2}(A)\to A\) be given by \(g(U)\coloneqq r\) when there exists \(k\in\{1,2\}\) such that \(U\in C_{k,r}\). By construction we have that \(f(a_1,a_2)=g(\{a_1,a_2\})\) for all \(a_1,a_2\in A\). We now show that \(\mathbf{A}\) is strongly fair. Given \(r\in A\) we have that \(f(a_1,a_2)=r\) with \(\{a_1,a_2\}\in A_1\) when \(\{a_1,a_2\}\in C_{1,r}\). There is only one way to obtain an ordered pair from \(\{a_1,a_2\}\) with \(a_1=a_2\) so \(\lvert f^{-1}(r)\cap A_1\rvert=1\lvert C_{1,r}\rvert=1\). Given \(r\in A\) we have that \(f(a_1,a_2)=r\) with \(\{a_1,a_2\}\in A_2\) when \(\{a_1,a_2\}\in C_{2,r}\). There are two ways to obtain an ordered pair from \(\{a_1,a_2\}\) with \(a_1\neq a_2\) so \(\lvert f^{-1}(r)\cap A_2\rvert=2\lvert C_{2,r}\rvert\). As each of the \(C_{2,r}\) have the same size we conclude that \(\mathbf{A}\) is strongly fair. Since \(m=2s+1\) for some \(s\in\N\) we have that \(m\ge3>2=n\) so \(\mathbf{A}\) is also nondegenerate. We see that an essentially polyadic, strongly fair, nondegenerate magma always exists when \(m\neq1\) is odd.
\end{proof}

Historically those games which have been played or described tend to be conservative but we did not need that assumption for our argument. We say a partition \(P\coloneqq\{P_i\}_{i\in I}\) of a set \(S\) is \emph{regular} when \(\lvert P_i\rvert=\lvert P_j\rvert\) for all \(i,j\in I\). Note that we gave a description of all possible essentially polyadic, strongly fair, nondegenerate magmas \(\mathbf{A}\) with \(m\neq1\) odd as any such magma will induce a regular partition \(C_1\) of \(\binom{A}{1}\) and a regular partition \(C_2\) of \(\binom{A}{2}\) and such a pair of partitions of \(\binom{A}{1}\) and \(\binom{A}{2}\) will yield a map \(f\colon A^2\to A\) with the desired properties. It is also the case that there is at least one conservative, essentially polyadic, strongly fair, and nondegenerate magma \(\mathbf{A}\) for each odd \(m\neq1\). We produce examples of these games in \autoref{sec:examples}.

We explore selection games for more than \(2\) simultaneous players, in particular those which we see as generalized Rock-Paper-Scissors games. We give a numerical constraint on which \(n\)-magmas of order \(m\) can be essentially polyadic, strongly fair, and nondegenerate and show that this constraint is sharp. After giving examples of such magmas for all possible pairs \((m,n)\) we go on to detail connections between RPS, tournament magmas, and hypertournaments. We proceed to count these magmas and study their automorphisms and congruences, concluding with the exhibition of an infinite family of finite simple magmas.

\section{RPS magmas}
\label{sec:rps_magmas}
The magmas we are interested in are those corresponding to selection games which have the four desirable properties possessed by Rock-Paper-Scissors. As in the preceding section it will benefit us to first examine the larger class of magmas obtained by dropping the conservativity axiom.

\begin{definition}[\(\prps\) magma]
Let \(\mathbf{A}\coloneqq(A,f)\) be an \(n\)-ary magma. When \(\mathbf{A}\) is essentially polyadic, strongly fair, and nondegenerate we say that \(\mathbf{A}\) is a \emph{\(\prps\) magma} (read ``\emph{pseudo-\(\rps\) magma}''). When \(\mathbf{A}\) is an \(n\)-magma of order \(m\in\N\) with these properties we say that \(\mathbf{A}\) is a \emph{\(\prps(m,n)\) magma}. We also use \(\prps\) and \(\prps(m,n)\) to indicate the classes of such magmas.
\end{definition}

Our first theorem generalizes directly to selection games with more than \(2\) players.

\begin{theorem}
\label{thm:numerical_condition}
Let \(\mathbf{A}\in\prps(m,n)\) and let \(\varpi(m)\) denote the least prime dividing \(m\). We have that
	\begin{align}\label{eq:numerical_condition}
		n<\varpi(m).
	\end{align}
Conversely, for each pair \((m,n)\) with \(m\neq1\) such that \(n<\varpi(m)\) there exists such a magma.
\end{theorem}

\begin{proof}
Since \(\mathbf{A}\) is nondegenerate we must have that \(m>n\).

Since \(\mathbf{A}\) is strongly fair we must have that \(\lvert f^{-1}(a)\cap A_k\rvert=\lvert f^{-1}(b)\cap A_k\rvert\) for all \(k\in\N\). As the \(m\) distinct sets \(f^{-1}(a)\cap A_k\) for \(a\in A\) partition \(A_k\) and are all the same size we require that \(m\mid\lvert A_k\rvert\). When \(k>n\) we have that \(A_k=\varnothing\) and obtain no constraint on \(m\).

When \(k\le n\) we have that \(A_k\) is nonempty. As we take \(\mathbf{A}\) to be essentially polyadic we have that \(f(x)=f(y)\) for all \(x,y\in A_k\) such that \(\{x_1,\dots,x_n\}=\{y_1,\dots,y_n\}\). Let \(\binom{A}{k}\) denote the collection of \(k\)-sets in \(A\). We require that \(m\mid\lvert\binom{A}{k}\rvert=\binom{m}{k}\) for all \(k\le n\).

Let
	\[
		d(m,n)\coloneqq\gcd\left(\left\{\,\binom{m}{k}\;\middle|\;1\le k\le n\,\right\}\right).
	\]
Since \(m\mid\binom{m}{k}\) for all \(k\le n\) we must have that \(m\mid d(m,n)\). Joris, Oestreicher, and Steinig showed that when \(m>n\) we have
	\[
		d(m,n)=\frac{m}{\lcm(\{\,k^{\varepsilon_k(m)}\mid1\le k\le n\,\})}
	\]
where \(\varepsilon_k(m)=1\) when \(k\mid m\) and \(\varepsilon_k(m)=0\) otherwise\cite[p.103]{joris}. Since we have that \(m\mid d(m,n)\) and \(d(m,n)\mid m\) it must be that \(m=d(m,n)\) and hence
	\[
		\lcm\left(\left\{\,k^{\varepsilon_k(m)}\;\middle|\;1\le k\le n\,\right\}\right)=1.
	\]
This implies that \(\varepsilon_k(m)=0\) for all \(2\le k\le n\). That is, no \(k\) between \(2\) and \(n\) inclusive divides \(m\). This is equivalent to having that no prime \(p\le n\) divides \(m\), which is in turn equivalent to having that \(n<\varpi(m)\), as desired.

It remains to show that such games \(\mathbf{A}\) exist when \(m\neq1\) and \(n<\varpi(m)\). By this assumption we have that \(k\nmid m\) whenever \(2\le k\le n\). Since
	\[
		\binom{m}{k}=\frac{m!}{(m-k)!k!}=m\frac{(m-1)\cdots(m-k+1)}{k(k-1)\cdots(2)}
	\]
and none of the nontrivial factors of \(k!\) divide \(m\) it must be that \(m\mid\binom{m}{k}\) for each \(2\le k\le n\). This implies that \(m\mid\lvert\binom{A}{k}\rvert\) for each \(k\le n\) so for each \(k\le n\) we can partition \(\binom{A}{k}\) into \(m\) subcollections \(C_k\coloneqq\{C_{k,r}\}_{r\in A}\) indexed on the \(m\) elements of \(A\), each with \(\lvert C_{k,r}\rvert=\frac{1}{m}\binom{m}{k}\). With respect to this collection of partitions \(C\coloneqq\{C_k\}_{1\le k\le n}\) we define an \(n\)-ary operation \(f\colon A^n\to A\) by \(f(a_1,\dots,a_n)\coloneqq r\) when \(\{a_1,\dots,a_n\}\in C_{k,r}\) for some \(k\in\{1,\dots,n\}\). This map is well-defined since each \(\{a_1,\dots,a_n\}\) contains exactly \(k\) distinct elements for some \(k\in\{1,\dots,n\}\) and thus belongs to a unique member of one of the partitions \(C_k\). In order to see that the resulting magma \(\mathbf{A}\coloneqq(A,f)\) is essentially polyadic let \(g\colon\pow_{\le n}(A)\to A\) be given by \(g(U)\coloneqq r\) when there exists \(k\in\{1,\dots,n\}\) such that \(U\in C_{k,r}\). By construction we have that \(f(a_1,\dots,a_n)=g(\{a_1,\dots,a_n\})\) for all \(a_1,\dots,a_n\in A\). We now show that \(\mathbf{A}\) is strongly fair. Given \(r\in A\) we have that \(f(a_1,\dots,a_n)=r\) with \((a_1,\dots,a_n)\in A_k\) when \(\{a_1,\dots,a_n\}\in C_{k,r}\). Note that the number of members of \(A_k\) whose coordinates form the set \(\{a_1,\dots,a_n\}\) is the same as the number of members of \(A_k\) whose coordinates form the set \(\{b_1,\dots,b_n\}\) for some other \((b_1,\dots,b_n)\in A_k\). This implies that each of the \(\lvert f^{-1}(r)\cap A_k\rvert\) have the same size for a fixed \(k\) and hence \(\mathbf{A}\) is strongly fair. To see that \(\mathbf{A}\) is nondegenerate observe that \(n<\varpi(m)\le m\).
\end{proof}

As in the case of \(n=2\) we did not use the conservativity axiom. We have given a description of all possible finite \(\prps\) magmas. To see this, note that any \(\prps\) magma of order \(m\) satisfying our numerical condition \eqref{eq:numerical_condition} will induce a regular partition \(C_k\) of \(\binom{A}{k}\) for each \(1\le k\le n\) and any such collection of partitions will yield a map \(f\colon A^n\to A\) with the desired properties. As before it is always possible to find conservative \(\prps(m,n)\) magmas when this numerical condition \eqref{eq:numerical_condition} is met. These magmas are those which possess all the desirable properties of the game RPS.

\begin{definition}[\(\rps\) magma]
Let \(\mathbf{A}\coloneqq(A,f)\) be an \(n\)-ary magma. When \(\mathbf{A}\) is conservative, essentially polyadic, strongly fair, and nondegenerate we say that \(\mathbf{A}\) is an \emph{\(\rps\) magma}. When \(\mathbf{A}\) is an \(n\)-magma of order \(m\in\N\) with these properties we say that \(\mathbf{A}\) is an \emph{\(\rps(m,n)\) magma}. We also use \(\rps\) and \(\rps(m,n)\) to indicate the classes of such magmas.
\end{definition}

More succinctly, \(\rps\) magmas are conservative \(\prps\) magmas. We proceed to exhibit members of this class.

\section{Examples of RPS magmas}
\label{sec:examples}
We give examples of \(\rps(m,n)\) magmas for all \(m\) and \(n\) satisfying the numerical constraint \eqref{eq:numerical_condition} of \autoref{thm:numerical_condition}. This shows that such pairs \((m,n)\) are precisely those for which such magmas exist. Our construction makes use of group actions and we first give a lemma in that direction.

\begin{definition}[\(k\)-extension of an action]
Given a group action \(\alpha\colon\mathbf{G}\to\permg(A)\) of \(\mathbf{G}\) on a set \(A\) and some \(1\le k\le\lvert A\rvert\) define for each \(s\in G\) a map \(\alpha_k(s)\colon \binom{A}{k}\to \binom{A}{k}\) by
	\[
		(\alpha_k(s))(U)\coloneqq\{\,(\alpha(s))(a)\mid a\in U\,\}.
	\]
The function \(\alpha_k\colon G\to \binom{A}{k}^{\binom{A}{k}}\) is called the \emph{\(k\)-extension} of \(\alpha\).
\end{definition}

\begin{lemma}
The \(k\)-extension of a group action \(\alpha\colon\mathbf{G}\to\permg(A)\) is a group action.
\end{lemma}

\begin{proof}
This is easily verified.
\end{proof}

Certain group actions on a set \(A\) yield families of magmas with universe \(A\). Recall that a group action \(\alpha\colon\mathbf{G}\to\permg(X)\) is called \emph{free} when for any \(s,t\in G\) and any \(x\in X\) we have that \((\alpha(s))(x)=(\alpha(t))(x)\) implies \(s=t\), \emph{transitive} when for any \(x,y\in X\) there exists some \(s\in G\) such that \((\alpha(s))(x)=y\), and \emph{regular} when \(\alpha\) is both free and transitive.

\begin{definition}[\(\alpha\)-action magma]
Fix a group \(\mathbf{G}\), a set \(A\), and some \(n<\lvert A\rvert\). Given a regular group action \(\alpha\colon\mathbf{G}\to\permg(A)\) such that each of the \(k\)-extensions of \(\alpha\) is free for \(1\le k\le n\) let \(\Psi_k\coloneqq\{\,\orb(U)\mid U\in\binom{A}{k}\,\}\) where \(\orb(U)\) is the orbit of \(U\) under \(\alpha_k\). Let \(\beta\coloneqq\{\beta_k\}_{1\le k\le n}\) be a sequence of choice functions \(\beta_k\colon\Psi_k\to \binom{A}{k}\) such that \(\beta_k(\psi)\in\psi\) for each \(\psi\in\Psi_k\). Let \(\gamma\coloneqq\{\gamma_k\}_{1\le k\le n}\) be a sequence of functions \(\gamma_k\colon\Psi_k\to A\) such that \(\gamma_k(\psi)\in\beta_k(\psi)\) for each \(\psi\in\Psi_k\). Let \(g\colon\pow_{\le n}(A)\to A\) be given by \(g(U)\coloneqq(\alpha(s))(\gamma_k(\psi))\) when \(U=(\alpha_k(s))(\beta_k(\psi))\). Define \(f\colon A^n\to A\) by \(f(a_1,\dots,a_n)\coloneqq g(\{a_1,\dots,a_n\})\). The \emph{\(\alpha\)-action magma} induced by \((\beta,\gamma)\) is \(\mathbf{A}\coloneqq(A,f)\).
\end{definition}

The function \(g\), and hence \(f\), is well-defined as we assume that each of the \(\alpha_k\) is free and hence there is a unique \(s\in G\) such that \(U=(\alpha_k(s))(\beta_k(\psi))\) for each \(U\in\binom{A}{k}\).

\begin{theorem}
\label{thm:action_magmas_are_rps}
Let \(\mathbf{A}\) be an \(\alpha\)-action magma induced by \((\beta,\gamma)\). We have that \(\mathbf{A}\in\rps\).
\end{theorem}

\begin{proof}
We show that \(\mathbf{A}\) is conservative. Let \(a_1,\dots,a_n\in A\) and define \(U\coloneqq\{a_1,\dots,a_n\}\in\binom{A}{k}\). Suppose that \(U\in\psi\in\Psi_k\) with \(U=(\alpha_k(s))(\beta_k(\psi))\). Observe that
	\[
		f(a_1,\dots,a_n)=g(U)=(\alpha(s))(\gamma_k(\psi)).
	\]
By assumption \(\gamma_k(\psi)\in\beta_k(\psi)\) so
	\[
		f(a_1,\dots,a_n)=(\alpha(s))(\gamma_k(\psi))\in(\alpha_k(s))(\beta_k(\psi))=U,
	\]
as desired.

By definition of \(f\) via \(g\colon\pow_{\le n}(A)\to A\) we have that \(\mathbf{A}\) is essentially polyadic.

In order to see that \(\mathbf{A}\) is strongly fair note that \(\lvert\psi\rvert=\lvert G\rvert\) for each \(\psi\in\Psi_k\) as we assume the action of \(\alpha_k\) on \(\binom{A}{k}\) to be free. For each orbit \(\psi\) we have that \(g\) takes on each value in \(A\) exactly once as the action of \(\alpha\) on \(A\) is assumed to be transitive. This shows that each orbit \(\psi\in\Psi_k\) contributes the same number of elements to each of the sets \(f^{-1}(a)\cap A_k\). It follows that \(\mathbf{A}\) is strongly fair.

By definition of an \(\alpha\)-action magma we must have that \(n<\lvert A\rvert\) so \(\mathbf{A}\) is nondegenerate.
\end{proof}

In order to have an \(\alpha\)-action magma for \(\alpha\colon\mathbf{G}\to\permg(A)\) we must have that \(\alpha\) is regular. Recall that every regular \(\mathbf{G}\)-action is isomorphic (in the category of \(\mathbf{G}\)-sets) to the left multiplication action \(L\colon\mathbf{G}\to\permg(G)\). Isomorphic actions determine equivalent orbits so without loss of generality we may only consider the left-multiplication actions of groups on themselves. Fortunately this class of actions is highly compatible with our construction.

\begin{proposition}
\label{thm:free_action_extension}
Let \(\mathbf{G}\) be a nontrivial finite group and let \(L\colon\mathbf{G}\to\permg(G)\) be the left-multiplication action. Whenever \(\gcd(k,\lvert G\rvert)=1\) we have that the \(k\)-extension \(L_k\) of \(L\) is free. In particular, when \(1\le k<\varpi(\lvert G\rvert)\) we have that the \(k\)-extension \(L_k\) of \(L\) is free.
\end{proposition}

\begin{proof}
In order to show that \(L_k\) is free it suffices to show that given \(s\in G\) and \(U\in\binom{G}{k}\) such that \((L_k(s))(U)=U\) we have that \(s=e\). If \((L_k(s))(U)=U\) then we can write \(U=\coprod_{i=1}^{r}\{s^ju_i\}_j\) as a disjoint union of the \(r\) orbits of elements of \(U\) under \(L(s)\). Each of these orbits has size exactly \(\lvert s\rvert\), which divides \(\lvert\mathbf{G}\rvert\). We find that \(\lvert s\rvert\) divides both \(\lvert U\rvert=k\) and \(\lvert\mathbf{G}\rvert\). When \(\gcd(k,\lvert G\rvert)=1\) this can only occur when \(\lvert s\rvert=1\). Thus, \(s=e\).
\end{proof}

Since we know that our numerical condition \eqref{eq:numerical_condition} for the existence of an \(\rps(m,n)\) magma must hold we have characterized all \(\mathbf{G}\)-actions which give rise to a finite \(\rps\) magma through this construction.

\begin{definition}[Regular \(\rps\) magma]
Let \(\mathbf{G}\) be a nontrivial finite group and fix \(n<\varpi(\lvert G\rvert)\). We denote by \(\mathbf{G}_n(\beta,\gamma)\) the \(L\)-action \(n\)-magma induced by \((\beta,\gamma)\), which we refer to as a \emph{regular \(\rps\) magma}.
\end{definition}

The games Rock-Paper-Scissors and Rock-Paper-Scissors-Spock-Lizard are isomorphic to regular \(\rps\) magmas. To obtain the classic Rock-Paper-Scissors take \(\mathbf{G}=\Z_3\) and \(n=2\). We have \(\binom{A}{1}=\{\{0\},\{1\},\{2\}\}\) and \(\binom{A}{2}=\{\{0,1\},\{0,2\},\{1,2\}\}\). Under the action of \(\Z_3\) we have
	\[
		\Psi_1=\{\{\{0\},\{1\},\{2\}\}\}
	\]
and
	\[
		\Psi_2=\{\{\{0,1\},\{0,2\},\{1,2\}\}\}.
	\]
There is only one orbit \(\psi_{1,1}\in\Psi_1\), for which we choose \(\beta_1(\psi_{1,1})\coloneqq\{0\}\) as a representative. There is also only one orbit \(\psi_{2,1}\in\Psi_2\), for which we choose \(\beta_2(\psi_{2,1})\coloneqq\{0,1\}\) as a representative. We choose \(\gamma_1(\psi_{1,1})\coloneqq0\) and \(\gamma_2(\psi_{2,1})\coloneqq1\). We have that \(\{0\}=0+\{0\}\), \(\{1\}=1+\{0\}\), \(\{2\}=2+\{0\}\), \(\{0,1\}=0+\{0,1\}\), \(\{1,2\}=1+\{0,1\}\), and \(\{0,2\}=2+\{0,1\}\). The resulting values of \(g\) are \(g(\{0\})=0\), \(g(\{1\})=1\), \(g(\{2\})=2\), \(g(\{0,1\})=1\), \(g(\{1,2\})=2\), and \(g(\{0,2\})=0\). The Cayley table for the operation \(f\) obtained from \(g\) is given in \autoref{fig:regular_rps}. Observe that under the identification \(0\mapsto r\), \(1\mapsto p\), and \(2\mapsto s\) the magma \((\Z_3)_2(\beta,\gamma)\) is the original game of RPS.

\begin{figure}[ht]
	\begin{tabular}{r|*{3}{c}}
		& \(0\) & \(1\) & \(2\) \\ \hline
		\(0\) & \(0\) & \(1\) & \(0\) \\
		\(1\) & \(1\) & \(1\) & \(2\) \\
		\(2\) & \(0\) & \(2\) & \(2\)
	\end{tabular}
	\caption{RPS as a regular \(\rps\) magma}
	\label{fig:regular_rps}
\end{figure}

To obtain the game Rock-Paper-Scissors-Spock-Lizard take \(\mathbf{G}=\Z_5\) and \(n=2\). In order to simplify notation we will write sets as strings. For example \(012\) denotes \(\{0,1,2\}\). The sets \(\binom{A}{k}\) are then \(\binom{A}{1}=\{0,1,2,3,4\}\) (where \(a\) represents the singleton set \(\{a\}\)) and \(\binom{A}{2}=\{01,02,03,04,12,13,14,23,24,34\}\). Under the action of \(\Z_5\) we have
	\[
		\Psi_1=\{\{0,1,2,3,4\}\}
	\]
and
	\[
		\Psi_2=\{\{01,12,23,34,40\},\{02,13,24,30,41\}\}.
	\]
There is only one orbit \(\psi_{1,1}\in\Psi_1\), for which we choose \(\beta_1(\psi_{1,1})\coloneqq0\) as a representative. There are two orbits in \(\Psi_2\), \(\psi_{2,1}\coloneqq\{01,12,23,34,40\}\) and \(\psi_{2,2}\coloneqq\{02,13,24,30,41\}\). We choose \(\beta_2(\psi_{2,1})\coloneqq01\) as a representative of \(\psi_{2,1}\) and \(\beta_2(\psi_{2,2})\coloneqq02\) as a representative of \(\psi_{2,2}\). We choose \(\gamma_1(\psi_{1,1})\coloneqq0\), \(\gamma_2(\psi_{2,1})\coloneqq1\), and \(\gamma_2(\psi_{2,2})\coloneqq2\). The Cayley table resulting from these choices is given in \autoref{fig:regular_rpssl}. Observe that under the identification \(0\mapsto r\), \(1\mapsto p\), \(2\mapsto s\), \(3\mapsto v\), and \(4\mapsto l\) the magma \((\Z_5)_2(\beta,\gamma)\) is the game of RPSSL.

\begin{figure}[ht]
	\begin{tabular}{r|*{5}{c}}
		& \(0\) & \(1\) & \(2\) & \(3\) & \(4\) \\ \hline
		\(0\) & \(0\) & \(1\) & \(0\) & \(3\) & \(0\) \\
		\(1\) & \(1\) & \(1\) & \(2\) & \(1\) & \(4\) \\
		\(2\) & \(0\) & \(2\) & \(2\) & \(3\) & \(2\) \\
		\(3\) & \(3\) & \(1\) & \(3\) & \(3\) & \(4\) \\
		\(4\) & \(0\) & \(4\) & \(2\) & \(4\) & \(4\)
	\end{tabular}
	\caption{RPSSL as a regular \(\rps\) magma}
	\label{fig:regular_rpssl}
\end{figure}

We now give an example of one of the games with \(3\) players. We know that such a game must have at least \(m=5\) items to choose from so we take \(\mathbf{G}=\Z_5\). We make the same choices for the orbits in \(\Psi_1\) and \(\Psi_2\) as in RPSSL but we must now examine
	\[
		\Psi_3=\{\{012,123,234,340,401\},\{013,124,230,341,402\}\}.
	\]
Choose \(\beta_3(\psi_{3,1})\coloneqq012\) as a representative of \(\psi_{3,1}\coloneqq\{012,123,234,340,401\}\) and \(\beta_3(\psi_{3,2})\coloneqq013\) as a representative of \(\psi_{3,2}\coloneqq\{013,124,230,341,402\}\). We choose \(\gamma_3(\psi_{3,1})\coloneqq0\) and \(\gamma_3(\psi_{3,2})\coloneqq0\). In \autoref{fig:rps53_first} and \autoref{fig:rps53_second} we give the Cayley table for the resulting ternary operation as five binary tables, one for \((x,y)\mapsto f(0,x,y)\), one for \((x,y)\mapsto f(1,x,y)\), one for \((x,y)\mapsto f(2,x,y)\), one for \((x,y)\mapsto f(3,x,y)\), and one for \((x,y)\mapsto f(4,x,y)\).

\begin{figure}[ht]
	\begin{tabular}{r|*{5}{c}}
		\(0\) & \(0\) & \(1\) & \(2\) & \(3\) & \(4\) \\ \hline
		\(0\) & \(0\) & \(1\) & \(0\) & \(3\) & \(0\) \\
		\(1\) & \(1\) & \(1\) & \(0\) & \(0\) & \(4\) \\
		\(2\) & \(0\) & \(0\) & \(0\) & \(2\) & \(4\) \\
		\(3\) & \(3\) & \(0\) & \(2\) & \(3\) & \(3\) \\
		\(4\) & \(0\) & \(4\) & \(4\) & \(3\) & \(0\)
	\end{tabular}
	\qquad
	\begin{tabular}{r|*{5}{c}}
		\(1\) & \(0\) & \(1\) & \(2\) & \(3\) & \(4\) \\ \hline
		\(0\) & \(1\) & \(1\) & \(0\) & \(0\) & \(4\) \\
		\(1\) & \(1\) & \(1\) & \(2\) & \(1\) & \(4\) \\
		\(2\) & \(0\) & \(2\) & \(2\) & \(1\) & \(1\) \\
		\(3\) & \(0\) & \(1\) & \(1\) & \(1\) & \(3\) \\
		\(4\) & \(4\) & \(4\) & \(1\) & \(3\) & \(4\)
	\end{tabular}
	\qquad
	\begin{tabular}{r|*{5}{c}}
		\(2\) & \(0\) & \(1\) & \(2\) & \(3\) & \(4\) \\ \hline
		\(0\) & \(0\) & \(0\) & \(0\) & \(2\) & \(4\) \\
		\(1\) & \(0\) & \(2\) & \(2\) & \(1\) & \(1\) \\
		\(2\) & \(0\) & \(2\) & \(2\) & \(3\) & \(2\) \\
		\(3\) & \(2\) & \(1\) & \(3\) & \(3\) & \(2\) \\
		\(4\) & \(4\) & \(1\) & \(2\) & \(2\) & \(2\)
	\end{tabular}
	\caption{Cayley tables for \(f(0,x,y)\), \(f(1,x,y)\), and \(f(2,x,y)\)}
	\label{fig:rps53_first}
\end{figure}

\begin{figure}[ht]
	\begin{tabular}{r|*{5}{c}}
		\(3\) & \(0\) & \(1\) & \(2\) & \(3\) & \(4\) \\ \hline
		\(0\) & \(3\) & \(0\) & \(2\) & \(3\) & \(3\) \\
		\(1\) & \(0\) & \(1\) & \(1\) & \(1\) & \(3\) \\
		\(2\) & \(2\) & \(1\) & \(3\) & \(3\) & \(2\) \\
		\(3\) & \(3\) & \(1\) & \(3\) & \(3\) & \(4\) \\
		\(4\) & \(3\) & \(3\) & \(2\) & \(4\) & \(4\)
	\end{tabular}
	\qquad
	\begin{tabular}{r|*{5}{c}}
		\(4\) & \(0\) & \(1\) & \(2\) & \(3\) & \(4\) \\ \hline
		\(0\) & \(0\) & \(4\) & \(4\) & \(3\) & \(0\) \\
		\(1\) & \(4\) & \(4\) & \(1\) & \(3\) & \(4\) \\
		\(2\) & \(4\) & \(1\) & \(2\) & \(2\) & \(2\) \\
		\(3\) & \(3\) & \(3\) & \(2\) & \(4\) & \(4\) \\
		\(4\) & \(0\) & \(4\) & \(2\) & \(4\) & \(4\)
	\end{tabular}
	\caption{Cayley tables for \(f(3,x,y)\) and \(f(4,x,y)\)}
	\label{fig:rps53_second}
\end{figure}

In practice it is easier to actually play by using the map \(g\colon\pow_{\le n}(A)\to A\) which shows that \(\mathbf{A}\) is essentially polyadic. We give those maps for RPS, RPSSL, and our example of an \(\rps(5,3)\) magma in \autoref{fig:rps_g}, \autoref{fig:rpssl_g}, and \autoref{fig:rps53_g}, respectively. The orbits are separated by vertical dividers.

\begin{figure}[ht]
	\begin{tabular}{r|ccc|ccc}
		\(U\) & \(0\) & \(1\) & \(2\) & \(01\) & \(12\) & \(20\) \\ \hline
		\(g(U)\) & \(0\) & \(1\) & \(2\) & \(0\) & \(1\) & \(2\)
	\end{tabular}
	\caption{The function \(g\) for RPS}
	\label{fig:rps_g}
\end{figure}

\begin{figure}[ht]
	\begin{tabular}{r|ccc|ccccc|ccccc}
		\(U\) & \(0\) & \(1\) & \(2\) & \(01\) & \(12\) & \(23\) & \(34\) & \(40\) & \(02\) & \(13\) & \(24\) & \(30\) & \(41\) \\ \hline
		\(g(U)\) & \(0\) & \(1\) & \(2\) & \(1\) & \(2\) & \(3\) & \(4\) & \(0\) & \(0\) & \(1\) & \(2\) & \(3\) & \(4\)
	\end{tabular}
	\caption{The function \(g\) for RPSSL}
	\label{fig:rpssl_g}
\end{figure}

\begin{figure}[ht]
	\centering
	\begin{tabular}{r|ccc|ccccc|ccccc}
		\(U\) & \(0\) & \(1\) & \(2\) & \(01\) & \(12\) & \(23\) & \(34\) & \(40\) & \(02\) & \(13\) & \(24\) & \(30\) & \(41\) \\ \hline
		\(g(U)\) & \(0\) & \(1\) & \(2\) & \(1\) & \(2\) & \(3\) & \(4\) & \(0\) & \(0\) & \(1\) & \(2\) & \(3\) & \(4\)
	\end{tabular}
	\begin{tabular}{r|ccccc|ccccc}
		\(U\) & \(012\) & \(123\) & \(234\) & \(340\) & \(401\) & \(013\) & \(124\) & \(230\) & \(341\) & \(402\) \\ \hline
		\(g(U)\) & \(0\) & \(1\) & \(2\) & \(3\) & \(4\) & \(0\) & \(1\) & \(2\) & \(3\) & \(4\)
	\end{tabular}
	\caption{The function \(g\) for an \(\rps(5,3)\) example}
	\label{fig:rps53_g}
\end{figure}

Not all \(\rps\) magmas can be obtained by the preceding construction. Consider the \(\rps(7,2)\) magma \(\mathbf{A}\) given by the table in \autoref{fig:rps72}. This magma corresponds to the graph Chamberland and Herman call HexagonalPyramid, which they show has automorphism group isomorphic to \(\Z_3\)\cite[p.7]{chamberland}. In \autoref{sec:automorphisms} we show that the automorphism group of \(\mathbf{G}_2(\beta,\gamma)\) must contain a copy of \(\mathbf{G}\). It follows that a regular \(\rps\) magma of order \(7\) cannot have an automorphism group with less than \(7\) elements so \(\mathbf{A}\) cannot be isomorphic to a regular \(\rps\) magma.

\begin{figure}[ht]
	\begin{tabular}{r|*{7}{c}}
		& \(0\) & \(1\) & \(2\) & \(3\) & \(4\) & \(5\) & \(6\) \\ \hline
		\(0\) & \(0\) & \(1\) & \(0\) & \(3\) & \(4\) & \(0\) & \(0\) \\
		\(1\) & \(1\) & \(1\) & \(2\) & \(1\) & \(1\) & \(5\) & \(6\) \\
		\(2\) & \(0\) & \(2\) & \(2\) & \(3\) & \(2\) & \(5\) & \(2\) \\
		\(3\) & \(3\) & \(1\) & \(3\) & \(3\) & \(4\) & \(3\) & \(6\) \\
		\(4\) & \(4\) & \(1\) & \(2\) & \(4\) & \(4\) & \(5\) & \(4\) \\
		\(5\) & \(0\) & \(5\) & \(5\) & \(3\) & \(5\) & \(5\) & \(6\) \\
		\(6\) & \(0\) & \(6\) & \(2\) & \(6\) & \(4\) & \(6\) & \(6\)  
	\end{tabular}
	\caption{An \(\rps(7,2)\) magma}
	\label{fig:rps72}
\end{figure}

\section{Tournaments and RPS magmas}
\label{sec:tournaments}
Now that we have specialized from \(\prps\) magmas to \(\rps\) magmas to regular \(\rps\) magmas and obtained some basic information we give more general context to our discussion. We detail the relationship between \(\rps\) magmas, tournament algebras, and the hypertournaments considered in hypergraph theory.

An \(\rps\) magma can be viewed as arising from a hypergraph with some additional data. Our hypergraphs are allowed to have infinitely many vertices.

\begin{definition}[Pointed hypergraph]
A \emph{pointed hypergraph} \(\mathbf{S}\coloneqq(S,\sigma,g)\) consists of a hypergraph \((S,\sigma)\) and a map \(g\colon\sigma\to S\) such that for each edge \(e\in\sigma\) we have that \(g(e)\in e\). The map \(g\) is called a \emph{pointing} of \((S,\sigma)\).
\end{definition}

We make use of the following family of hypergraphs.

\begin{definition}[\(n\)-complete hypergraph]
Given a set \(S\) we denote by \(\mathbf{S}_n\) the \emph{\(n\)-complete hypergraph} whose vertex set is \(S\) and whose edge set is \(\bigcup_{k=1}^n\binom{S}{k}\).
\end{definition}

Alternatively in the nondegenerate case that \(\lvert S\rvert\ge n\) the \(n\)-complete hypergraph on \(S\) is the \((n-1)\)-skeleton of the simplex with vertex set \(S\) viewed as an abstract simplicial complex, minus the \((-1)\)-cell \(\varnothing\).

\begin{definition}[Hypertournament]
An \emph{\(n\)-hypertournament} is a pointed hypergraph \(\mathbf{T}\coloneqq(T,\tau,g)\) where \((T,\tau)=\mathbf{S}_n\) for some set \(S\).
\end{definition}

We can think of the vertices of an \(n\)-hypertournament as players where for each collection of at most \(n\) players those players participate in a game for which a single winner is decided. Given an \(n\)-hypertournament \(\mathbf{T}\coloneqq(T,\tau,g)\), \(V=\{v_1,\dots,v_k\}\in\binom{T}{k}\) for \(1\le k\le n-1\), and \(u\in T\) such that \(g(\{u,v_1,\dots,v_k\})=u\) we say that \(u\) \emph{dominates} \(V\) and write \(u\to V\) to indicate this. When \(V=\{v\}\) is a singleton we sometimes write \(u\to v\) rather than \(u\to V\).

In the case that \(n=2\) we recover the classical graph-theoretic \emph{tournament}.

Our definition of hypertournament differs from that of Surmacs\cite{surmacs}, for example. The benefit of our definition is that there is a natural algebraic formulation of this concept.

\begin{definition}[Hypertournament magma]
Given an \(n\)-hypertournament \(\mathbf{T}\) the \emph{hypertournament magma} obtained from \(\mathbf{T}\coloneqq(T,\tau,g)\) is the \(n\)-magma \(\mathbf{A}\coloneqq(T,f)\) where for \(u_1,\dots,u_n\in T\) we define
	\[
		f(u_1,\dots,u_n)\coloneqq g(\{u_1,\dots,u_n\}).
	\]
\end{definition}

Note that hypertournament magmas are precisely those magmas which are conservative and essentially polyadic, allowing us to give the following equivalent definition.

\begin{definition}[Hypertournament magma]
A \emph{hypertournament magma} is an \(n\)-magma which is conservative and essentially polyadic.
\end{definition}

In 1965 Hedrl{\'i}n and Chv{\'a}tal introduced the \(n=2\) case of a hypertournament magma, which we refer to as a \emph{tournament magma}. Tournaments, both as graphs and magmas, have been studied extensively\cite{crvenkovic}. Earlier treatments are more likely to characterize tournament magmas as those magmas which are commutative and conservative (and idempotent, although this is redundant).

Let \(\rps_n\) denote the class of \(n\)-ary \(\rps\) magmas, let \(\prps_n\) denote the class of \(n\)-ary \(\prps\) magmas, and let \(\tour_n\) denote the class of all \(n\)-ary hypertournament magmas. We give the containment relationships for these classes.

\begin{proposition}
Let \(n>1\). We have that \(\rps_n\subsetneq\prps_n\), \(\rps_n\subsetneq\tour_n\), and neither of \(\prps_n\) and \(\tour_n\) contains the other. Moreover, \(\rps_n=\prps_n\cap\tour_n\).
\end{proposition}

\begin{proof}
Given an \(\rps\) \(n\)-magma \(\mathbf{A}\coloneqq(A,f)\) and some \(\sigma\in\perm(A)\) we can define a \(\prps\) \(n\)-magma \(\mathbf{A}_\sigma\coloneqq(A,f_\sigma)\) where for \(x_1,\dots,x_n\in A\) we set \(f_\sigma(x_1,\dots,x_n)\coloneqq\sigma(f(x_1,\dots,x_n))\). The resulting magma \(\mathbf{A}_\sigma\) will always belong to \(\prps_n\) but will not in general belong to \(\rps_n\). As by definition we have \(\rps_n\subset\prps_n\) this shows that \(\rps_n\subsetneq\prps_n\).

We also have by definition that \(\rps_n\subset\tour_n\). Consider the magma \(\mathbf{A}\coloneqq(\{a,b\},f)\) where given \(x\in A^n\) we define \(f(x)\coloneqq a\) when at least one component of \(x\) is \(a\) and \(f(x)\coloneqq b\) otherwise. We have that \(\mathbf{A}\) belongs to \(\tour_n\) but does not belong to \(\rps_n\) so \(\rps_n\subsetneq\tour_n\).

We have immediately from the definitions that \(\rps_n=\prps_n\cap\tour_n\).
\end{proof}

Given \(n>1\) we have that \(\prps_n\) and \(\rps_n\) are not varieties. We included a nondegeneracy condition in our definitions of \(\prps\) and \(\rps\) in order to avoid repeatedly disavowing certain trivial situations. Since \(\prps_n\) and \(\rps_n\) do not contain any trivial members we have immediately that \(\prps_n\) and \(\rps_n\) are not varieties. Even if we expand these classes to include trivial algebras they still do not form varieties, as we now show.

The classes \(\prps_n\) and \(\rps_n\) do not support taking subalgebras for \(n>1\).

\begin{proposition}
Let \(n>1\). There exists a magma \(\mathbf{A}\) belonging to \(\rps_n\) and a subalgebra \(\mathbf{B}\le\mathbf{A}\) such that \(\mathbf{B}\) is not a \(\prps\) magma.
\end{proposition}

\begin{proof}
Since \(\rps\) magmas are conservative we have that any subset of an \(\rps\) magma's universe is a subuniverse. Take \(\mathbf{A}\coloneqq(\Z_p)_n(\beta,\gamma)\) for a prime \(p>n\) and any valid choice of \((\beta,\gamma)\). Take \(\mathbf{B}\) to be any subalgebra of \(\mathbf{A}\) of order \(p-1\). Observe that \(\mathbf{A}\) is a member of \(\prps_n\). However, the order of \(\mathbf{B}\) is even and \(n>1\) so it cannot be that \(n<\varpi(p-1)\). This shows that \(\mathbf{B}\) is not a \(\prps\) magma.
\end{proof}

Note that both the original RPS magma and the magma for the French variant are contained in the magma for RPSSL so in the \(n=2\) case we have seen that an \(\rps\) magma may have some subalgebras which are \(\rps\) magmas and some which are not. Note that \(\{r,s,l\}\) is a subuniverse for RPSSL and the corresponding subalgebra satisfies the numerical condition \eqref{eq:numerical_condition} necessary for \(\prps_2\) magmas (being binary and of order \(3\)), yet the corresponding subalgebra fails to be strongly fair and as such does not belong to \(\prps_2\).

The classes \(\rps_n\) for \(n>1\) fail to be varieties for another reason. They are as far from being closed under products as possible.

\begin{proposition}
Let \(\mathbf{A}\) and \(\mathbf{B}\) be nontrivial, essentially polyadic, conservative \(n\)-magmas. The magma \(\mathbf{A}\times\mathbf{B}\) is not conservative. In particular, the product of two \(\rps\) \(n\)-magmas is not an \(\rps\) \(n\)-magma.
\end{proposition}

\begin{proof}
We show that \(\mathbf{A}\times\mathbf{B}\) cannot be conservative. Let \(\mathbf{A}\coloneqq(A,f)\) and let \(\mathbf{B}\coloneqq(B,g)\). Let \(x_1,x_2\in A\) be distinct and let \(y_1,y_2\in A\) be distinct with \(f(x_1,x_2,x_2,\dots,x_2)=x_1\) and \(g(y_1,y_2,y_2,\dots,y_2)=y_2\). We have that
	\[
		(f\times g)((x_1,y_1),(x_2,y_2),(x_2,y_2),\dots,(x_2,y_2))=(x_1,y_2)
	\]
so \(\mathbf{A}\times\mathbf{B}\) cannot be conservative.
\end{proof}

Given \(n>1\) we have that the class \(\tour_n\) is not a variety. That \(\tour_2\) is not closed under taking products has already been demonstrated\cite[p.98]{crvenkovic}. Given a pair of magmas \(\mathbf{A}\coloneqq(A,f)\) and \(\mathbf{B}\coloneqq(B,g)\) in \(\tour_2\) such that \(\mathbf{A}\times\mathbf{B}\notin\tour_2\) construct magmas \(\mathbf{A}'\coloneqq(A,f')\) and \(\mathbf{B}'\coloneqq(B,g')\) in \(\tour_n\) such that \(f'(x_1,\dots,x_n)=f(u,v)\) when \(\{x_1,\dots,x_n\}=\{u,v\}\) and similarly for \(g'\). Since \(\mathbf{A}\times\mathbf{B}\) is not in \(\tour_2\) it cannot be the case that \(\mathbf{A}'\times\mathbf{B}'\) belongs to \(\tour_n\).

Let \(\mathcal{T}_n\coloneqq\mathbf{V}(\tour_n)\) be the variety generated by all \(n\)-hypertournaments, let \(\mathcal{P}_n\coloneqq\mathbf{V}(\prps_n)\) be the variety generated by all \(n\)-ary \(\prps\) magmas, and let \(\mathcal{R}_n\coloneqq\mathbf{V}(\rps_n)\) be the variety generated by all \(n\)-ary \(\rps\) magmas. We study the relationships between these varieties.

\begin{theorem}
\label{thm:finite_regular_nary_rps_generate}
Let \(n>1\). We have that \(\mathcal{T}_n=\mathcal{R}_n\). Moreover \(\mathcal{T}_n\) is generated by the class of finite regular \(\rps_n\) magmas.
\end{theorem}

In order to make this argument we refer to the hypergraph-theoretic analogue of regular \(\rps\) magmas.

\begin{definition}[Regular balanced hypertournament]
A \emph{regular balanced hypertournament} is a hypertournament \(\mathbf{T}\coloneqq(T,\tau,g)\) such that there exists a regular \(\rps\) magma \(\mathbf{G}_n(\beta,\gamma)\) where \(G=T\) and \(g(u_1,\dots,u_n)=u_i\) when \(f(u_1,\dots,u_n)=u_i\) in \(\mathbf{G}_n(\beta,\gamma)\).
\end{definition}

We give an alternative characterization of regular \(\rps\) magmas.

\begin{definition}[\((\beta,n)\)-chirality]
Given \(n>1\) and a nontrivial finite group \(\mathbf{G}\) with \(n<\varpi(\lvert G\rvert)\) let \(\Psi_k\) denote the collection of orbits of \(\binom{G}{k}\) under the \(k\)-extension of the left-multiplication action of \(\mathbf{G}\) on itself. Given a sequence of choice functions \(\beta=\{\beta_1,\dots,\beta_n\}\) where \(\beta_k\colon\Psi_k\to\binom{G}{k}\) we say that a sequence of maps \(\gamma=\{\gamma_1,\dots,\gamma_n\}\) such that \(\gamma_k\colon\Psi_k\to G\) with \(\gamma_k(\psi)\in\beta_k(\psi)\) for each \(\psi\in\Psi_k\) is a \emph{\((\beta,n)\)-chirality} of \(\mathbf{G}\). We denote by \(\chir_n(\mathbf{G},\beta)\) the collection of all \((\beta,n)\)-chiralities of \(\mathbf{G}\).
\end{definition}

For each fixed \((\beta,n)\) the members \(\gamma\) of \(\chir_n(\mathbf{G},\beta)\) are precisely the data for each possible regular \(n\)-ary \(\rps\) magma obtained from \(\mathbf{G}\).

\begin{definition}[Obverse \(k\)-set]
Given \(n>1\), a nontrivial finite group \(\mathbf{G}\) with \(n<\varpi(\lvert G\rvert)\), \(1\le k\le n-1\), and \(U,V\in\binom{G\setminus\{e\}}{k}\) we say that \(V\) is an \emph{obverse} of \(U\) when \(U=\{a_1,\dots,a_k\}\) and there exists some \(a_i\in U\) such that \(V=\{a_i^{-1}\}\cup\{\,a_i^{-1}a_j\mid i\neq j\,\}\). We denote by \(\obv(U)\) the set consisting of all obverses \(V\) of \(U\), as well as \(U\) itself.
\end{definition}

Note that any \(U\in\binom{G\setminus\{e\}}{k}\) has exactly \(k\) distinct obverses since given any obverse \(V\) of \(U\) we have that \((V\cup\{e\})=a_i^{-1}(U\cup\{e\})\) and the left-multiplication action of \(\mathbf{G}\) on itself extends to a free action on \(\binom{G}{k+1}\) for \(1\le k\le n-1\). That is, the obverses of \(U\) are the nonidentity elements in the members of \(\orb(U\cup\{e\})\setminus(U\cup\{e\})\) which contain \(e\).

\begin{definition}[\(n\)-sign function]
Given \(n>1\) and a nontrivial group \(\mathbf{G}\) with \(n<\varpi(\lvert G\rvert)\) let \(\sgn_n(\mathbf{G})\) denote the set of all choice functions on
	\[
		\left\{\,\obv(U)\;\middle|\;(\exists k\in\{1,\dots,n-1\})\left(U\in\binom{G\setminus\{e\}}{k}\right)\,\right\}.
	\]
We refer to a member \(\lambda\in\sgn_n(\mathbf{G})\) as an \emph{\(n\)-sign function} on \(\mathbf{G}\).
\end{definition}

We can turn a \((\beta,n)\)-chirality into an \(n\)-sign function. We need to know that given a regular \(\rps\) magma \(\mathbf{G}_n(\beta,\gamma)\) whose basic operation is \(f\) and distinct \(a_1,\dots,a_k\in G\setminus\{e\}\) we have that
	\[
		f(e,\dots,e,a_1^{-1},a_1^{-1}a_2,\dots,a_1^{-1}a_k)=a_1^{-1}f(e,\dots,e,a_1,\dots,a_k).
	\]
This follows from the left-multiplication embedding of \(\mathbf{G}\) in \(\autg(\mathbf{G}_n(\beta,\gamma))\) given in \autoref{prop:left_multiplication_automorphism}.

\begin{definition}[\((\beta,n)\)-signor]
Define the \emph{\((\beta,n)\)-signor} \(\zeta_{\beta,n}\colon\chir_n(\mathbf{G},\beta)\to\sgn_n(\mathbf{G})\) as follows. Given \(\gamma\in\chir_n(\mathbf{G},\beta)\), let \(f_\gamma\) denote the basic operation of \(\mathbf{G}_n(\beta,\gamma)\). Given \(U=\{a_1,\dots,a_k\}\in\binom{G\setminus\{e\}}{k}\) for some \(1\le k\le n-1\) define
	\[
		(\zeta_{\beta,n}(\gamma))(\obv(U))\coloneqq U
	\]
when
	\[
		f_\gamma(e,\dots,e,a_1,\dots,a_k)=e.
	\]
\end{definition}

We choose from among \(U\) and its obverses the subset \(V\) of \(G\setminus\{e\}\) such that \(e\) dominates \(V\). The maps \(\zeta_{\beta,n}(\gamma)\) are well-defined because if we exchange \(\{a_1,\dots,a_k\}\) for one of its obverses, say \(\{a_1^{-1},a_1^{-1}a_2,\dots,a_1^{-1}a_k\}\), and use that left-multiplication by \(a_1^{-1}\) is an automorphism of \(\mathbf{G}_n(\beta,\gamma)\) (or use that \(f(e,\dots,e,a_1^{-1},a_1^{-1}a_2,\dots,a_1^{-1}a_k)=a_1^{-1}f(e,\dots,e,a_1,\dots,a_k)\) with \(f\) the basic operation of \(\mathbf{G}_n(\beta,\gamma)\)) we have that
	\[
		(\zeta_{\beta,n}(\gamma))(\obv(\{a_1,\dots,a_k\}))=\{a_1,\dots,a_k\}
	\]
when
	\[
		f_\gamma(e,\dots,e,a_1,\dots,a_k)=e,
	\]
which is equivalent to
	\[
		f_\gamma(e,\dots,e,a_1^{-1},a_1^{-1}a_2,\dots,a_1^{-1}a_k)=a_1,
	\]
which is equivalent to having that
	\[
		(\zeta_{\beta,n}(\gamma))(\obv(\{a_1^{-1},a_1^{-1}a_2,\dots,a_1^{-1}a_k\}))\neq\{a_1^{-1},a_1^{-1}a_2,\dots,a_1^{-1}a_k\},
	\]
so no other representative of \(\obv(U)\) can be the value of \((\zeta_{\beta_n}(\gamma))(\obv(U))\).

For a particular choice of \(\beta\) the \(n\)-sign functions on a group \(\mathbf{G}\) are equivalent to the \((\beta,n)\)-chiralities.

\begin{lemma}
The map \(\zeta_{\beta,n}\colon\chir_n(\mathbf{G},\beta)\to\sgn_n(\mathbf{G})\) is a bijection.
\end{lemma}

\begin{proof}
Let \(\zeta\coloneqq\zeta_{\beta,n}\). Suppose that \(\gamma,\gamma'\in\chir_n(\mathbf{G},\beta)\) with \(\zeta(\gamma)=\zeta(\gamma')\). A regular \(\rps\) magma operation is determined by choosing for each \(1\le k\le n\) and each orbit of \(\binom{G}{k}\) the \(k\)-set \(U\) for which \(g(U)=e\). Having that \(\zeta(\gamma)=\zeta(\gamma')\) says that these choices are the same for both \(\mathbf{G}_n(\beta,\gamma)\) and \(\mathbf{G}_n(\beta,\gamma')\). Taking \(f_\gamma\) and \(f_{\gamma'}\) to be the basic operations of \(\mathbf{G}_n(\beta,\gamma)\) and \(\mathbf{G}_n(\beta,\gamma')\), respectively, it follows that \(f_\gamma=f_{\gamma'}\). If \(\gamma\) and \(\gamma'\) were to differ then there would be some orbit \(\psi\in\Psi_k\) for some \(1\le k\le n\) for which \(\gamma_k(\psi)\neq\gamma'_k(\psi)\). This would imply that \(f_\gamma\neq f_{\gamma'}\), so it must be that \(\zeta(\gamma)=\zeta(\gamma')\) implies that \(\gamma=\gamma'\). That is, \(\zeta\) is injective.

The map \(\zeta\) is also surjective. Given an \(n\)-sign function \(\lambda\in\sgn(\mathbf{G})\), \(\psi\in\Psi_k\), and \(\beta_k(\psi)=\{a_1,\dots,a_k\}\) set \(U\coloneqq\{\,a_1^{-1}a_i\mid i\neq1\,\}\) and define \(\gamma_k\colon\Psi_k\to G\) by \(\gamma_k(\psi)\coloneqq a_1\) when \(\lambda(\obv(U))=U\). We claim that for this choice of \(\gamma_k\) for each \(k\) we have \(\zeta(\gamma)=\lambda\). To see this, observe that given \(V=\{b_1,\dots,b_{k-1}\}\in\binom{G\setminus\{e\}}{k-1}\) for \(1\le k\le n\) with \(\psi=\orb(\{e,b_1,\dots,b_{k-1}\})\) and \(\beta_k(\psi)=\{a_1,\dots,a_k\}\) we have that \(\{e,b_1,\dots,b_{k-1}\}=c\{a_1,\dots,a_k\}\) for some \(c\in G\). Suppose that \(e=ca_1\) and \(b_i=ca_{i+1}\) for \(1\le i\le k-1\). We have that \(c=a_1^{-1}\) and \(b_i=a_1^{-1}a_{i+1}\). Taking \(U=\{\,a_1^{-1}a_i\mid i\neq1\,\}\) it follows that \((\zeta(\gamma))(\obv(U))=U\) if and only if \(f_\gamma(e,\dots,e,a_1^{-1}a_2,\dots,a_1^{-1}a_k)=e\), which occurs if and only if \(f_\gamma(e,\dots,e,a_1,\dots,a_k)=a_1\). This is equivalent to having that \(\lambda(\obv(U))=U\) by definition of \(\gamma\). We find that \(\zeta(\gamma)=\lambda\), so \(\zeta\) is indeed surjective and hence a bijection.
\end{proof}

This lemma says that for \(n\)-ary regular \(\rps\) magmas we can always work with choice functions from \(\sgn_n(\mathbf{G})\) rather than \(\chir_n(\mathbf{G},\beta)\). For any \(\beta\) and \(\beta'\) we have that \(\mathbf{G}_n(\beta,\zeta_{\beta,n}^{-1}(\lambda))=\mathbf{G}_n(\beta',\zeta_{\beta',n}^{-1}(\lambda))\) so \(\mathbf{G}_n(\lambda)\) is well-defined without reference to a particular choice of orbit representatives \(\beta\).

We give the proof of \autoref{thm:finite_regular_nary_rps_generate}.

\begin{proof}
Since \(n\)-hypertournament magmas are conservative it is the case that \(\tour_n\models\epsilon\) for some identity \(\epsilon\) in \(m\) variables exactly when each of the \(n\)-hypertournament magmas of order \(m\) models \(\epsilon\). Since
	\[
		\mathcal{T}_n=\mathbf{V}(\tour_n)=\operatorname{Mod}(\operatorname{Id}(\tour_n))
	\]
we have that \(\mathcal{T}_n\) consists of all models of those identities which hold in the finite \(n\)-hypertournament magmas. We find that \(\mathcal{T}_n\) is generated by the finite members of \(\tour_n\). This is the same argument as the classical one for the binary case\cite[p.99]{crvenkovic}.

We show that every finite \(n\)-hypertournament embeds into a finite regular balanced hypertournament. Since we know that the finite hypertournaments generate \(\mathcal{T}_n\) this will establish that the finite regular balanced hypertournaments alone generate \(\mathcal{T}_n\).

Let \(\mathbf{T}\coloneqq(T,\tau,g)\) be a finite hypertournament. Consider the group \(\mathbf{G}\coloneqq\bigoplus_{u\in T}\Z_{\alpha_u}\) where for each \(u\) we have that \(\alpha_u\) is a natural other than \(1\) and \(n<\varpi(\alpha_u)\). Identify \(u\in T\) with a generator of \(\Z_{\alpha_u}\). Define an \(n\)-sign function \(\lambda\in\sgn_n(\mathbf{G})\) as follows. Given distinct vertices \(u_1,\dots,u_k\in T\) for some \(1\le k\le n\) where \(g(\{u_1,\dots,u_k\})=u_1\) we define
	\[
		\lambda(\obv(\{\,u_i-u_1\mid i\neq1\,\}))\coloneqq\{\,u_i-u_1\mid i\neq1\,\}.
	\]
The function \(\lambda\) may take on any values for other arguments. By the previous lemma \(\lambda\) yields a regular \(\rps\) magma \(\mathbf{G}_n(\lambda)\). It follows that \(\mathbf{T}\) embeds into the corresponding regular balanced hypertournament.

Since the variety \(\mathcal{T}_n\) is generated by the finite regular members of \(\rps_n\) we have that \(\mathcal{T}_n\le\mathcal{R}_n\). Conversely we have that \(\rps_n\subset\tour_n\) so \(\mathcal{R}_n\le\mathcal{T}_n\). We find that \(\mathcal{T}_n=\mathcal{R}_n\).
\end{proof}

We have a few corollaries arising from taking specific values of \(\alpha_u\) in the preceding proof.

\begin{corollary}
The finite regular \(\rps\) magmas of the form \((\Z_{\kappa(n)}^m)(\lambda)\) where \(\kappa(n)\) is the least prime strictly greater than \(n\) generate \(\mathcal{T}_n\).
\end{corollary}

\begin{proof}
Take each \(\alpha_u\coloneqq\kappa(n)\) in the previous argument.
\end{proof}

\begin{corollary}
Let \(p_k\) indicate the \(k^{\text{th}}\) prime and suppose that \(\kappa(n)=p_\ell\). Define \(\alpha(m)\coloneqq\prod_{k=\ell}^{m+\ell-1}p_k\). The finite regular \(\rps\) magmas of the form \((\Z_{\alpha(m)})_n(\lambda)\) generate \(\mathcal{T}_n\).
\end{corollary}

\begin{proof}
Order the \(u\in T\) as \(\{u_i\}_{i=1}^m\) and let \(\alpha_{u_i}\coloneqq p_{i+\ell-1}\). We have that \(\mathbf{G}=\bigoplus_{i=1}^m\Z_{p_{i+\ell-1}}\cong\Z_{\alpha(m)}\).
\end{proof}

Our last corollary is of a more combinatorial nature.

\begin{corollary}
\label{cor:hypertournament_embedding}
Every finite \(n\)-hypertournament of order \(m\) is contained in a finite regular balanced hypertournament of order \(\kappa(n)^m\) .
\end{corollary}

\begin{proof}
Again, take each \(\alpha_u\coloneqq\kappa(n)\) in the previous argument.
\end{proof}

We describe a setting in which this corollary is much weaker than an analogous result we have in the binary case.

\begin{definition}[Balanced hypertournament]
\label{balanced_hypertournament}
A \emph{balanced hypertournament} is a hypertournament \(\mathbf{T}\coloneqq(T,\tau,g)\) such that for all \(u,v\in T\) and all \(k\in\N\) we have
	\[
		\lvert\{\,W\in\tau\mid\lvert W\rvert=k\text{ and }g(W)=u\,\}\rvert=\lvert\{\,W\in\tau\mid\lvert W\rvert=k\text{ and }g(W)=v\,\}\rvert.
	\]
\end{definition}

This is to say that all vertices in a balanced hypertournament dominate the same number of edges, where ``number'' is meant in the sense of cardinality for infinite collections of vertices \(T\). Just as hypertournament magmas are in natural correspondence with hypertournaments we have that \(\rps\) magmas are in natural correspondence with balanced hypertournaments.

In the \(n=2\) case we have a different embedding result. Our hypergraph-theoretic notation is extraneous here so we revert to the graph-theoretic formulation where a tournament on \(r\) vertices is a directed graph \(\mathbf{T}\coloneqq(T,\tau)\) whose underlying graph is the complete simple graph on \(r\) vertices and a balanced tournament is a tournament in which each vertex dominates the same number of other vertices.

\begin{lemma}
Every finite tournament is a subgraph of a finite balanced tournament.
\end{lemma}

\begin{proof}
Let \(\mathbf{T}\coloneqq(T,\tau)\) be a finite tournament with \(r\) vertices. Define
	\[
		\Phi\coloneqq\{\,u_1\mid u\in T\,\}\cup\{\,u_2\mid u\in T\,\}\cup\{\eta\}
	\]
where \(\eta\) is a new vertex distinct from the \(u_i\). Define a relation \(\phi\subset\Phi^2\) as follows. Given \(u_i\in\Phi\) we set \((u_i,v_i)\in\phi\) when \((u,v)\in\tau\). We set \((u_i,v_j)\in\phi\) for \(u\neq v\) and \(i\neq j\) when \((v,u)\in\tau\). We set \((u_1,u_2)\in\phi\), \((u_2,\eta)\in\phi\), and \((\eta,u_1)\in\phi\). We claim that \(\mathbf{\Phi}\coloneqq(\Phi,\phi)\) is a balanced tournament.

It is immediate that \(\mathbf{\Phi}\) is a tournament. To see that \(\mathbf{\Phi}\) is balanced observe that \(\mathbf{\Phi}\) consists of two copies of \(\mathbf{T}\) along with the new vertex \(\eta\). The vertex \(u_i\) dominates some \(k\) vertices in the \(i\) copy of \(\mathbf{T}\) and dominates the \(n-k-1\) vertices in the \(j\neq i\) copy of \(\mathbf{T}\) corresponding to the \(n-k-1\) vertices in \(\mathbf{T}\) which dominate \(u\), giving \(n-1\) vertices dominated by \(u_i\). We have that \(u_1\) dominates \(u_2\) and is in turn dominated by \(\eta\). It follows that each \(u_1\) has out-degree \(n\). Similarly, each \(u_2\) dominates \(\eta\) and is in turn dominated by \(u_1\) so each \(u_2\) has out-degree \(n\). The new vertex \(\eta\) dominates the \(n\) vertices \(u_1\) and is dominated by the vertices \(u_2\) so \(\eta\) also has out-degree \(n\). We find that \(\mathbf{\Phi}\) is a balanced tournament. The vertices \(u_i\) for a fixed \(i\) induce a subgraph of \(\mathbf{\Phi}\) isomorphic to \(\mathbf{T}\).
\end{proof}

Since every regular balanced hypertournament is a balanced hypertournament we have that our earlier \autoref{cor:hypertournament_embedding} implies that every finite tournament of order \(m\) can be embedded in a finite balanced tournament of order \(\kappa(2)^m=3^m\). In order to prove that \(\mathcal{T}_n=\mathcal{R}_n\) all we needed was an embedding of any finite \(n\)-hypertournament into a finite balanced \(n\)-hypertournament. An \(n\)-ary generalization of the preceding lemma would suffice to prove \autoref{thm:finite_regular_nary_rps_generate} with a tighter bound on the size of the balanced hypertournament an arbitrary order \(k\) hypertournament embeds in.

This bound has algebraic significance because it helps us search for identities. By the proof that the finite tournament magmas generate \(\mathcal{T}_2\) (loc. cit.) and our previous lemma we have that \(\mathcal{T}_2\models\alpha(x_1,\dots,x_k)\approx\beta(x_1,\dots,x_k)\) if and only if \(\alpha\approx\beta\) is modeled by each binary \(\rps\) magma of order \(2k+1\). If we use instead \autoref{cor:hypertournament_embedding} we can only show that \(\mathcal{T}_2\models\alpha(x_1,\dots,x_k)\approx\beta(x_1,\dots,x_k)\) if and only if \(\alpha\approx\beta\) is modeled by each binary \(\rps\) magma of order \(3^k\).

An identical analysis takes place in the \(n\)-ary case. Let \(h_n(k)\) denote the least natural such that each \(n\)-hypertournament of order \(k\) is contained in some balanced \(n\)-hypertournament of order \(h_n(k)\). For example, we know that \(h_2(k)\le2k+1\). We have that \(\mathcal{T}_n\models\alpha(x_1,\dots,x_k)\approx\beta(x_1,\dots,x_k)\) if and only if \(\alpha\approx\beta\) is modeled by each \(n\)-ary \(\rps\) magma of order \(h_n(k)\). In general we know that \(h_n(k)\le\kappa(n)^k\).

We leave open the problem of generalizing our previous lemma to \(n>2\), with the hope that such a generalization will give a subexponential bound for \(h_n(k)\).

Since we know the relationship between \(\mathcal{T}_n\) and \(\mathcal{R}_n\) we turn to the variety \(\mathcal{P}_n\). We know that \(\mathcal{P}_n\models f(x,\dots,x,y)\approx f(x,y,\dots,y)\) so \(\mathcal{P}_n\) has a nontrivial equational theory. It turns out that \(\mathcal{T}_n\) is properly contained in \(\mathcal{P}_n\). In order to show this we first exhibit a member of \(\mathcal{P}_2\) which will be useful to us.

Let \(\mathbf{B}\coloneqq(B,g)\) be the canonical \(\rps(3,2)\) magma corresponding to the original RPS game. Let \(\sigma\in\perm(B)\) be given by \(\sigma(r)\coloneqq p\), \(\sigma(p)\coloneqq s\), and \(\sigma(s)\coloneqq r\). Define a new magma \(\mathbf{B}_\sigma\coloneqq(B,g_\sigma)\) where for \(x,y\in B\) we set \(g_\sigma(x,y)\coloneqq\sigma(g(x,y))\). The Cayley table for \(\mathbf{B}_\sigma\) is given in \autoref{fig:permuted_table}. Note that \(\mathbf{B}_\sigma\) is essentially polyadic, strongly fair, and nondegenerate and as such belongs to \(\prps_2\) and hence \(\mathcal{P}_2\).

\begin{figure}[ht]
	\begin{tabular}{r|*{3}{c}}
		& \(r\) & \(p\) & \(s\) \\ \hline
		\(r\) & \(p\) & \(s\) & \(p\) \\
		\(p\) & \(s\) & \(s\) & \(r\) \\
		\(s\) & \(p\) & \(r\) & \(r\)
	\end{tabular}
	\caption{Cayley table for \(\mathbf{B}_\sigma\)}
	\label{fig:permuted_table}
\end{figure}

We use this magma to prove our final theorem of this section.

\begin{theorem}
We have \(\mathcal{T}_n<\mathcal{P}_n\) for \(n>1\).
\end{theorem}

\begin{proof}
Since \(\rps_n\subset\prps_n\) and \(\mathcal{T}_n=\mathcal{R}_n\) we have that \(\mathcal{T}_n\le\mathcal{P}_n\). Given \(\mathbf{A}\in\rps_n\) with \(\mathbf{A}\coloneqq(A,f)\) we have that \(\mathbf{A}'\in\rps_2\) where \(\mathbf{A}'\coloneqq(A,f')\) is given by \(f'(x,y)\coloneqq f(x,y,\dots,y)\). Let \(\alpha(x,y)\) denote the term \(f(x,y,\dots,y)\). It follows that since \(\mathcal{T}_2\models(x(yz))z\approx x((xy)z)\)\cite[p.115]{crvenkovic} we have that
	\[
		\mathcal{T}_n\models\alpha(\alpha(x,\alpha(y,z)),z)\approx\alpha(x,\alpha(\alpha(x,y),z)).
	\]

Let \(\mathbf{A}\coloneqq(\Z_{\kappa(n)})_n(\lambda)\) for some \(n\)-sign function \(\lambda\). Note that \(\mathbf{A}'\) must contain a copy of the canonical \(\rps(3,2)\) magma, for given two distinct elements \(u\) and \(v\) with \(u\to v\) we have that \(u\) is dominated by \(\frac{\kappa(n)-1}{2}\) elements other than \(u\) and \(v\) while \(v\) dominates \(\frac{\kappa(n)-1}{2}\) elements other than \(u\) and \(v\). There then exists some \(w\) with \(u\to v\to w\to u\). Permute the outputs of the basic operation of \(\mathbf{A}\) to obtain an algebra \(\mathbf{A}_\sigma\) analogous to our construction of \(\mathbf{B}_\sigma\) from \(\mathbf{B}\) as depicted in \autoref{fig:permuted_table}. Note that \(\mathbf{A}_\sigma\) contains an isomorphic copy of \(\mathbf{B}_\sigma\). We have that \(\mathbf{B}_\sigma\not\models(x(yz))z\approx x((xy)z)\) for
	\[
		(r(ps))s=(rr)s=ps=r\neq p=rr=r(ss)=r((rp)s)
	\]
in \(\mathbf{B}_\sigma\). This implies that
	\[
		\mathbf{A}_\sigma\not\models\alpha(\alpha(x,\alpha(y,z)),z)\approx\alpha(x,\alpha(\alpha(x,y),z))
	\]
and hence \(\mathcal{T}_n<\mathcal{P}_n\).
\end{proof}

\section{Counting RPS and PRPS magmas}
Given an order \(m\) and an arity \(n\) we determine the number of \(\prps(m,n)\) magmas of a given order on a fixed set of size \(m\). We denote the number of \(\prps(m,n)\) magmas on a fixed set of size \(m\) by \(\lvert\prps(m,n)\rvert\). We subsequently count how many regular \(\rps\) magmas \(\mathbf{G}_n(\lambda)\) exist given a finite group \(\mathbf{G}\). We denote the number of such magmas by \(\lvert\rps(\mathbf{G},n)\rvert\). Finally, we count the total number of \(\rps\) \(n\)-magmas on a fixed set of size \(m\), which we denote by \(\lvert\rps(m,n)\rvert\).

In order to perform these counts we must determine the number of regular partitions of \(\binom{A}{k}\) into \(m\) subsets of size \(\frac{1}{m}\lvert\binom{A}{k}\rvert\). That is, we must give the number \(\mathcal{B}(m,k)\) of regular partitions of a set of size \(\lvert\binom{A}{k}\rvert=\binom{m}{k}\) into \(m\) subsets of size \(\frac{1}{m}\binom{m}{k}\). In order to do this we give the number \(\mathscr{P}(m,s)\) of regular partitions of a set of size \(s\) into \(m\) subsets of size \(\frac{s}{m}\).

\begin{lemma}
Let \(m,s\in\N\) with \(m\mid s\). We have that
	\[
		\mathscr{P}(m,s)=\frac{1}{m!}\prod_{\ell=0}^{m-1}\binom{s-\ell\frac{s}{m}}{\frac{s}{m}}.
	\]
\end{lemma}

\begin{proof}
Let \(X\) be a set of size \(s\). By assumption \(m\mid s\) so we can produce a regular partition of \(X\) into \(m\) subsets. We do this in the following manner. First choose \(\frac{s}{m}\) elements from the \(s\) elements in \(X\) to form the first class of the partition, then choose \(\frac{s}{m}\) elements from the \(s-\frac{s}{m}\) elements in \(X\) not already chosen to form the second class of the partition, and so on. There are
	\[
		\prod_{\ell=0}^{m-1}\binom{s-\ell\frac{s}{m}}{\frac{s}{m}}
	\]
ways to specify a partition in this way. Since partitions are unordered we have counted each regular partition of \(X\) exactly \(m!\) times, once for each order we can place on the members of a partition. The result follows.
\end{proof}

The following special case will be relevant to us.

\begin{corollary}
Let \(m,k\in\N\) and suppose that \(m\mid\binom{m}{k}\). We have that
	\[
		\mathcal{B}(m,k)=\frac{1}{m!}\prod_{\ell=0}^{m-1}\binom{\binom{m}{k}-\ell\frac{1}{m}\binom{m}{k}}{\frac{1}{m}\binom{m}{k}}.
	\]
\end{corollary}

\begin{proof}
Take \(s=\binom{m}{k}\) in the previous lemma.
\end{proof}

We are now ready to compute \(\lvert\prps(m,n)\rvert\).

\begin{theorem}
Let \(m,n\in\N\) with \(m\neq1\) and \(n<\varpi(m)\). We have that
	\[
		\lvert\prps(m,n)\rvert=\prod_{k=1}^n\prod_{\ell=0}^{m-1}\binom{\binom{m}{k}-\ell\frac{1}{m}\binom{m}{k}}{\frac{1}{m}\binom{m}{k}}.
	\]
\end{theorem}

\begin{proof}
Let \(A\) be a fixed set with \(m\) elements. The members of \(\prps(m,n)\) with \(A\) as their universe are in bijection with regular ordered partitions of \(\binom{A}{k}\) into \(m\) subsets for \(1\le k\le n\). That is, in order to obtain a member of \(\prps(m,n)\) on \(A\) first choose for each \(1\le k\le n\) a partition of \(\binom{A}{k}\) into \(m\) subsets. There are \(\mathcal{B}(m,k)\) ways to do this for any given \(k\). Choose an order on the selected partitions. Since our partitions are partitions of \(\binom{A}{k}\) into \(m\) subsets there are \(m!\) ways to order each partition. This ordering corresponds to assigning to each equivalence class an element of \(A\). It follows that there are a total of
	\begin{align*}
		\prod_{k=1}^nm!\mathcal{B}(m,k) &= \prod_{k=1}^nm!\left(\frac{1}{m!}\prod_{\ell=0}^{m-1}\binom{\binom{m}{k}-\ell\frac{1}{m}\binom{m}{k}}{\frac{1}{m}\binom{m}{k}}\right)\\
		&= \prod_{k=1}^n\prod_{\ell=0}^{m-1}\binom{\binom{m}{k}-\ell\frac{1}{m}\binom{m}{k}}{\frac{1}{m}\binom{m}{k}}
	\end{align*}
ways to make these choices. Each such choice corresponds to the basic operation of a \(\prps(m,n)\) magma and each \(\prps(m,n)\) with universe \(A\) is uniquely described by such a choice.
\end{proof}

The value of \(\lvert\rps(\mathbf{G},n)\rvert\) is an even more direct computation.

\begin{theorem}
\label{thm:rps_G_n_count}
Let \(m,n\in\N\) with \(m\neq1\) and \(n<\varpi(m)\). Given a group \(\mathbf{G}\) of order \(m\) we have that
	\[
		\lvert\rps(\mathbf{G},n)\rvert=\prod_{k=1}^nk^{\frac{1}{m}\binom{m}{k}}.
	\]
\end{theorem}

\begin{proof}
Observe that for each \(1\le k\le n\) we have that \(\binom{G}{k}\) breaks up into \(\frac{1}{m}\binom{m}{k}\) orbits under the action of \(\mathbf{G}\). Fix a choice of orbit representatives \(\beta\). For each orbit \(\psi\) of \(\binom{G}{k}\) we must choose one of the \(k\) elements of \(\beta_k(\psi)\) to be \(\gamma_k(\psi)\). For a given \(k\) there are thus \(k^{\frac{1}{m}\binom{m}{k}}\) ways to make this choice, yielding a total of
	\[
		\prod_{k=1}^nk^{\frac{1}{m}\binom{m}{k}}
	\]
choices for \(\mathbf{G}_n(\beta,\gamma)\), as claimed.
\end{proof}

It seems to be more challenging to give an elementary formula for \(\lvert\rps(m,n)\rvert\). Given a polynomial \(p\in\Z[x_1,\dots,x_n]\) and a monomial \(x\coloneqq x_1^{s_1}\cdots x_n^{s_n}\) we write \(\mathcal{C}(x,p)\) to indicate the coefficient of \(x\) in \(p\). Given a finite set \(S=\{z_1,\dots,z_s\}\) of variables denote by \(\sum S\) the polynomial \(z_1+\dots+z_s\).

\begin{theorem}
Let \(m,n\in\N\) such that \(m\neq1\) and \(n<\varpi(m)\). Take \(A\coloneqq\{x_1,\dots,x_m\}\), let \(y_k\coloneqq(x_1\cdots x_m)^{\frac{1}{m}\binom{m}{k}}\), and define
	\[
		p_k\coloneqq\prod_{S\in\binom{A}{k}}\sum S.
	\]
We have that
	\[
		\lvert\rps(m,n)\rvert=\prod_{k=1}^n\mathcal{C}(y_k,p_k).
	\]
\end{theorem}

\begin{proof}
Each \(\rps(m,n)\) magma with universe \(A\) corresponds to a pointing \(g\) of \(\mathbf{A}_n\), the \(n\)-complete hypergraph with vertex set \(A\), such that \(\lvert g^{-1}(x_i)\cap\binom{A}{k}\rvert=\lvert g^{-1}(x_j)\cap\binom{A}{k}\rvert\) for each \(i\) and \(j\). Fixing some \(k\), the number of such ways to choose an element from each \(k\)-set in \(A\) is \(\mathcal{C}(y_k,p_k)\) as there is one factor of \(p_k\) for each \(k\)-set in \(A\) and we require that each \(x_i\) is chosen exactly \(\frac{1}{m}\binom{m}{k}\) times from among these factors by taking the \(y_k\) coefficient. Taking the product we obtain the total number of ways to choose a pointing \(g\).
\end{proof}

This result is the \(n\)-ary generalization of a formula given without a literature citation on the On-Line Encyclopedia of Integer Sequences for the number of labeled balanced tournaments\cite{oeisA007079}, which is also the quantity \(\lvert\rps(m,2)\rvert\). An asymptotic formula has been obtained via analytic methods\cite{mckay}.

\section{Automorphisms of regular RPS magmas}
\label{sec:automorphisms}
We describe the automorphism groups of regular \(\rps\) magmas.

\begin{proposition}
\label{prop:left_multiplication_automorphism}
Let \(\mathbf{A}\coloneqq\mathbf{G}_n(\lambda)\) be a regular \(\rps\) magma. There is a canonical embedding of \(\mathbf{G}\) into \(\autg(\mathbf{A})\).
\end{proposition}

The preceding proposition is an obvious extension of Cayley's Theorem from elementary group theory.

We would like to know whether there are any other automorphisms of regular \(\rps\) magmas. By counting we know that there must be regular \(\rps\) magmas with automorphism groups larger than \(\mathbf{G}\).

\begin{proposition}
For each arity \(n\in\N\) with \(n\neq1\) and each group \(\mathbf{G}\) of composite order \(m\in\N\) with \(n<\varpi(m)\) there exists a regular \(\rps(m,n)\) magma \(\mathbf{A}\coloneqq\mathbf{G}_n(\lambda)\) such that \(\lvert\autg(\mathbf{A})\rvert>\lvert\mathbf{G}\rvert\).
\end{proposition}

\begin{proof}
Suppose towards a contradiction that for each magma \(\mathbf{A}\coloneqq\mathbf{G}_n(\lambda)\) we have that \(\autg(\mathbf{A})\cong\mathbf{G}\). Each isomorphism class of magmas of the form \(\mathbf{G}_n(\beta,\gamma)\) must contribute
	\[
		\frac{\lvert\perm(G)\rvert}{\lvert G\rvert}=\frac{m!}{m}=(m-1)!
	\]
to the value of \(\lvert\rps(\mathbf{G},n)\rvert\) so \((m-1)!\) divides \(\lvert\rps(\mathbf{G},n)\rvert\). Since we take \(m\) to be composite we have that \(\varpi(m)\mid(m-1)!\) and hence \(\varpi(m)\mid\lvert\rps(\mathbf{G},n)\rvert\). As \(\lvert\rps(\mathbf{G},n)\rvert=\prod_{k=1}^nk^{\frac{1}{m}\binom{m}{k}}\) and each of the factors \(k^{\frac{1}{m}\binom{m}{k}}\) is a natural number which is a power of \(k\le n<\varpi(m)\) we have that \(\varpi(m)\nmid\lvert\rps(\mathbf{G},n)\rvert\). This is a contradiction.
\end{proof}

We have a similar result for groups of prime order.

\begin{proposition}
For each arity \(n\in\N\) and each odd prime \(p\) such that \(1\neq n\le p-2\) there exists a regular \(\rps(p,n)\) magma \(\mathbf{A}\coloneqq(\Z_p)_n(\lambda)\) such that \(\lvert\autg(\mathbf{A})\rvert>\lvert\mathbf{G}\rvert\).
\end{proposition}

\begin{proof}
Let \(r\) be a primitive root modulo \(p\) and define \(\varphi\colon\Z_p\to\Z_p\) by \(\varphi(x)\coloneqq rx\). We claim that there is a \(\lambda\in\sgn_n(\Z_p)\) for which \(\varphi\in\aut((\Z_p)_n(\lambda))\).

Let \(\Z_p^*\) denote the nonzero members of \(\Z_p\). Suppose towards a contradiction that \(U\in\binom{\Z_p^*}{k}\) for some \(1\le k\le n-1\) with \(\varphi(U)\in\obv(U)\). Writing the members of \(U\) as powers of \(r\) we have that
	\[
		U=\{r^{s_1},\dots,r^{s_k}\}.
	\]
Since \(\varphi(U)\in\obv(U)\) we can take
	\[
		\varphi(U)=\{-r^{s_1},-r^{s_1}+r^{s_2},\dots,-r^{s_1}+r^{s_k}\}.
	\]
By the definition of \(\varphi\) we also have that
	\[
		\varphi(U)=\{r^{s_1+1},\dots,r^{s_k+1}\}.
	\]
It follows that
	\[
		\varphi^2(U)=\{-r^{s_1+1},-r^{s_1+1}+r^{s_2+1},\dots,-r^{s_1+1}+r^{s_k+1}\}.
	\]
so \(\varphi^2(U)\in\obv(\varphi(U))=\obv(U)\). This implies that \(\obv(U)\) is closed under the action of \(\varphi\). Note that \(\varphi\) has order \(p-1\) while \(\lvert\obv(U)\rvert=k+1\le n\le p-2\). The only way this can occur is if the action of \(\varphi\) on \(\obv(U)\) is trivial, but this is impossible.

Since the action of \(\varphi\) on the collections of obverse \(k\)-sets from \(\Z_p^*\) for \(1\le k\le n-1\) is free we can choose from each \(\orb_\varphi(\obv(U))\) a representative collection of obverse \(k\)-sets, say \(\obv(U)\), and from that representative collection we can choose a \(k\)-set, say \(U\), for which we set \(\lambda(\obv(U))\coloneqq U\). Extending this choice by the action of \(\varphi\) we produce a \(\lambda\in\sgn_n(\Z_p)\) such that \(\varphi\in\aut((\Z_p)_n(\lambda))\).
\end{proof}

The case of a \((p-1)\)-ary regular \(\rps\) magma of prime order \(p\) is different.

\begin{proposition}
For each odd prime \(p\) and any \(\lambda\in\sgn_{p-1}(\Z_p)\) we have that \(\autg((\Z_p)_{p-1}(\lambda))\cong\Z_p\).
\end{proposition}

\begin{proof}
In this proof we use additive notation so that \(L_a(x)\coloneqq a+x\). Suppose towards a contradiction that there exists some \(\varphi\in\aut((\Z_p)_{p-1}(\lambda))\) such that for each \(a\in\Z_p\) we have that \(\varphi\neq L_a\). Let \(\varphi'\coloneqq L_{-\varphi(0)}\circ\varphi\). It follows that \(\varphi'\in\aut((\Z_p)_{p-1}(\lambda))\) with \(\varphi'(0)=0\) and \(\varphi'\neq L_0=\id_{\Z_p}\). Since \(\varphi'\) is nontrivial and fixes \(0\) there must be some nontrivial cycle of length at most \(p-1\) in \(\Z_p^*\) under the action of \(\varphi'\). Suppose that \(a_1,\dots,a_k\) is such a cycle and let \(f\) denote the basic operation of \((\Z_p)_{p-1}(\lambda)\). Relabeling if necessary we find that
	\[
		f(a_1,\dots,a_1,a_2,\dots,a_k)=a_1
	\]
and hence
	\begin{align*}
		f(a_1,\dots,a_1,a_2,\dots,a_k) &= f(\varphi'(a_k),\dots,\varphi'(a_k),\varphi'(a_1),\dots,\varphi'(a_{k-1}))\\
		&= \varphi'(f(a_k,\dots,a_k,a_1,\dots,a_{k-1}))\\
		&= \varphi'(f(a_1,\dots,a_1,a_2,\dots,a_k))\\
		&= \varphi'(a_1)\\
		&\neq a_1,
	\end{align*}
a contradiction.
\end{proof}

This allows us to count the magmas \((\Z_p)_{p-1}(\lambda)\) up to isomorphism.

\begin{corollary}
Given an odd prime \(p\) the number of isomorphism classes of magmas of the form \((\Z_p)_{p-1}(\lambda)\) is
	\[
		\prod_{k=1}^{p-1}k^{\frac{1}{p}\binom{p}{k}-1}.
	\]
\end{corollary}

\begin{proof}
This follows directly from \autoref{thm:rps_G_n_count} and that we now know
	\[
		\lvert\autg((\Z_p)_{p-1}(\lambda))\rvert=p
	\]
for any \(\lambda\in\sgn_{p-1}(\Z_p)\).
\end{proof}

Another class of automorphisms of \(\mathbf{G}_n(\lambda)\) comes from those automorphisms of \(\mathbf{G}\) which are compatible with \(\lambda\).

\begin{definition}[\(\lambda\)-automorphism]
Given a group \(\mathbf{G}\) and an \(n\)-sign function \(\lambda\in\sgn_n(\mathbf{G})\) we say that \(\varphi\in\aut(\mathbf{G})\) is a \emph{\(\lambda\)-automorphism} when for any \(U\in\binom{G\setminus\{e\}}{k}\) for \(1\le k\le n-1\) we have that \(\lambda(\obv(U))=U\) implies that \(\lambda(\obv(\varphi(U)))=\varphi(U)\).
\end{definition}

We denote by \(\autg_\lambda(\mathbf{G})\) the subgroup of \(\autg(\mathbf{G})\) consisting of all \(\lambda\)-automorphisms of \(\mathbf{G}\) and by \(\aut_\lambda(\mathbf{G})\) the set of all \(\lambda\)-automorphisms of \(\mathbf{G}\).

\begin{proposition}
\label{prop:lambda_automorphism}
Given a group \(\mathbf{G}\) and an \(n\)-sign function \(\lambda\in\sgn_n(\mathbf{G})\) there is a canonical embedding \(\autg_\lambda(\mathbf{G})\) into \(\autg(\mathbf{G}_n(\lambda))\).
\end{proposition}

\begin{proof}
Take \(\varphi\in\aut_\lambda(\mathbf{G})\) and \(\mathbf{A}\coloneqq\mathbf{G}_n(\lambda)\). Let \(a_1,\dots,a_n\in A\) with
	\[
		\{a_1,\dots,a_n\}\in\psi\in\Psi_k.
	\]
Suppose that \(f(a_1,\dots,a_n)=a_1\). This occurs if and only if \(\lambda(\obv(U))=U\) where \(U\coloneqq\{\,a_1^{-1}a_i\mid i\neq1\,\}\). We have that \(\lambda(\obv(U))=U\) if and only if \(\lambda(\obv(\varphi(U)))=\varphi(U)\). This in turn is equivalent to having that \(f(\varphi(a_1),\dots,\varphi(a_n))=\varphi(a_1)\). It follows that
	\[
		\varphi(f(a_1,\dots,a_n))=f(\varphi(a_1),\dots,\varphi(a_n)).
	\]
Thus, \(\varphi\) is an automorphism of \(\mathbf{A}\).
\end{proof}

Those \(n\)-sign functions for which all conjugations are automorphisms will be of interest to us. We denote by \(c_b\colon G\to G\) the conjugation map given by \(c_b(a)\coloneqq bab^{-1}\) and by \(\operatorname{\mathbf{Inn}}(\mathbf{G})\) the inner automorphism group of \(\mathbf{G}\).

\begin{definition}[Correlated \(n\)-sign function]
We say that an \(n\)-sign function \(\lambda\in\sgn_n(\mathbf{G})\) is \emph{correlated} when \(\operatorname{\mathbf{Inn}}(\mathbf{G})\le\autg_\lambda(\mathbf{G})\).
\end{definition}

We can construct all correlated \(n\)-sign functions on a group \(\mathbf{G}\) as follows. Observe that if \(U_1\coloneqq\{a_1,\dots,a_k\}\) and \(U_2\coloneqq\{a_1^{-1},a_1^{-1}a_2,\dots,a_1^{-1}a_k\}\) are obverse \(k\)-sets in \(\mathbf{G}\) and \(c_b\) is an inner automorphism of \(\mathbf{G}\) then we have that
	\[
		c_b(U_1)=\{c_b(a_1),\dots,c_b(a_k)\}
	\]
while
\begin{align*}
	c_b(U_2) &= \{c_b(a_1^{-1}),c_b(a_1^{-1}a_2),\dots,c_b(a_1^{-1}a_k)\}\\
	&= \{(c_b(a_1))^{-1},(c_b(a_1))^{-1}c_b(a_2),\dots,(c_b(a_1))^{-1}c_b(a_k)\}.
\end{align*}
Thus, if \(U_1\) and \(U_2\) are obverses then \(c_b(U_1)\) and \(c_b(U_2)\) will be obverses, as well. We now follow a familiar pattern. Consider the orbits of the sets \(\obv(U)\) under this action of \(\operatorname{\mathbf{Inn}}(\mathbf{G})\). Choose one representative set of obverses \(\obv(U)\) from each orbit and then for each orbit choose a member from its representative. Let that member be \(\lambda(\obv(U))\) and use the action of \(\operatorname{\mathbf{Inn}}(\mathbf{G})\) to extend these choices to an \(n\)-sign function \(\lambda\) defined on all sets of obverses in \(\mathbf{G}\).

It is possible to produce this extension by the action of \(\operatorname{\mathbf{Inn}}(\mathbf{G})\) because if \(U_1\) and \(U_2\) are obverse \(k\)-sets for \(k<\varpi(\lvert G\rvert)\) then it cannot be the case that \(c_b(U_1)=U_2\) for some \(b\in G\setminus\{e\}\). If this were the case then we would have that \(c_b|_{\obv(U_1)}\) is a permutation of \(\obv(U_1)\), which has \(k-1\) elements. It would follow that
	\[
		\lvert c_b|_{\obv(U_1)}\rvert\mid\lvert c_b\rvert\mid\lvert G\rvert,
	\]
but \(\lvert c_b|_{\obv(U_1)}\rvert\) must be divisible by some prime which is at most \(k-1<\varpi(\lvert G\rvert)\), a contradiction.

We give an explicit example of this using the nonabelian group of order \(21\). Let \(\F_7\) denote the field with \(7\) elements and let \(\F_7^*\coloneqq\F_7\setminus\{0\}\). We can take this nonabelian group \(\mathbf{G}\) of order \(21\) to be the group with universe
	\[
		G\coloneqq\{\,a^2x+b\mid(a,b)\in\F_7^*\times\F_7\,\}
	\]
whose multiplication is given by composition of polynomials in \(x\). We consider \(1\)-sets in \(G\) in order to construct a correlated \(\lambda\in\sgn_2(\mathbf{G})\). We identify the \(1\)-set \(\{a\}\) with \(a\) in order to simplify our notation.

\begin{figure}[ht]
	\begin{tabular}{*{3}{c}}
		Orbit of obverse sets & Representative set \(\obv(U)\) & \(\lambda(\obv(U))\) \\ \hline
		\(\{\,\{x+b,x-b\}\mid b\in\F_7^*\,\}\) & \(\{x+1,x+6\}\) & \(x+1\) \\
		\(\{\,\{2x+b,4x+3b\}\mid b\in\F_7\,\}\) & \(\{2x,4x\}\) & \(2x\)
	\end{tabular}
	\caption{A nontrivial correlated sign function}
	\label{fig:correlated_succinct}
\end{figure}

Using the action of \(\operatorname{\mathbf{Inn}}(\mathbf{G})\) we obtain the \(\lambda\in\sgn_2(\mathbf{G})\) given in \autoref{fig:correlated_succinct} and \autoref{fig:correlated_detailed}.

\begin{figure}[ht]
	\begin{tabular}{*{2}{c}}
		\(\obv(U)\) & \(\lambda(\obv(U))\) \\ \hline
		\(\{x+1,x+6\}\) & \(x+1\) \\
		\(\{x+2,x+5\}\) & \(x+2\) \\
		\(\{x+3,x+4\}\) & \(x+3\) \\
		\(\{2x,4x\}\) & \(2x\) \\
		\(\{2x+1,4x+3\}\) & \(2x+1\) \\
		\(\{2x+2,4x+6\}\) & \(2x+2\) \\
		\(\{2x+3,4x+2\}\) & \(2x+3\) \\
		\(\{2x+4,4x+5\}\) & \(2x+4\) \\
		\(\{2x+5,4x+1\}\) & \(2x+5\) \\
		\(\{2x+6,4x+4\}\) & \(2x+6\)
	\end{tabular}
	\caption{A nontrivial correlated sign function}
	\label{fig:correlated_detailed}
\end{figure}

\begin{theorem}
The semidirect product \(\mathbf{G}\rtimes\autg_\lambda(\mathbf{G})\) embeds in \(\autg(\mathbf{G}_n(\lambda))\). In particular, when \(\lambda\) is correlated we have that \(\mathbf{G}\rtimes\operatorname{\mathbf{Inn}}(\mathbf{G})\) embeds in \(\autg(\mathbf{G}_n(\lambda))\).
\end{theorem}

\begin{proof}
This follows directly from \autoref{prop:left_multiplication_automorphism} and \autoref{prop:lambda_automorphism}.
\end{proof}

\section{A group action constraint}
Our construction of regular \(\rps\) magmas and our result constraining the pairs \((m,n)\) for which \(\rps(m,n)\) magmas may exist can be combined to give a constraint on when the extensions of a regular group action are free.

\begin{theorem}
Suppose that \(\alpha\colon\mathbf{G}\to\permg(A)\) is a regular action of a group \(\mathbf{G}\) of order \(m\neq1\). The \(k\)-extensions \(\alpha_k\) of \(\alpha\) are all free for \(1\le k\le n\) if and only if \(n<\varpi(m)\).
\end{theorem}

\begin{proof}
We already argued that if \(n<\varpi(m)\) then each of the \(\alpha_k\) are free for \(1\le k\le n\) in the proof of \autoref{thm:free_action_extension} for the case that \(A=G\). The argument in the general case is identical.

Conversely, suppose that each of the \(\alpha_k\) are free for \(1\le k\le n\). By \autoref{thm:action_magmas_are_rps} we have that the corresponding \(\alpha\)-action magma is an \(\rps(m,n)\) magma. By \autoref{thm:numerical_condition} we have that \(n<\varpi(m)\).
\end{proof}

This implies that a particular extension is not free.

\begin{corollary}
Suppose that \(\alpha\colon\mathbf{G}\to\perm(A)\) is a regular action of a group \(\mathbf{G}\) of odd order \(m\neq1\). We have that the \(\varpi(m)\)-extension \(\alpha_{\varpi(m)}\) of \(\alpha\) is not free.
\end{corollary}

\section{Congruences of finite regular RPS magmas}
We study the structure of the congruence lattices of finite regular \(\rps\) magmas.

\begin{theorem}
Let \(\theta\in\con(\mathbf{A})\) for a regular \(\rps(m,n)\) magma \(\mathbf{A}\coloneqq\mathbf{G}_n(\lambda)\). Given any \(a\in A\) we have that \(a/\theta=aH\) for some subgroup \(\mathbf{H}\le\mathbf{G}\).
\end{theorem}

\begin{proof}
Consider the principal congruence \(\theta\coloneqq\cg(\{(e,a)\})\) generated by \((e,a)\). Since \(\mathbf{A}\) is conservative \(\theta\) has at most one nontrivial equivalence class, which is the equivalence class containing \(e\) and \(a\). We show that this equivalence class must contain \(\sg^{\mathbf{G}}(\{a\})\), the cyclic subgroup of \(\mathbf{G}\) generated by \(a\).

In order to do this we note that since \(\mathbf{G}\) has odd order it is \(2\)-divisible and the map \(x\mapsto x^2\) is bijective. Let \(\sqrt{a}\) denote the inverse image of \(a\) under this map.

Previously we wrote \(u\to V\) to indicate that \(u\) dominates \(V\) in a hypertournament. We expand on this notation slightly by writing \(u\to v\) rather than \(u\to\{v\}\) in the case that \(V=\{v\}\) is a singleton. Either \(e\to\sqrt{a}\), in which case applying the left-multiplication automorphism \(L_{\sqrt{a}}\) for \(\sqrt{a}\) yields that \(\sqrt{a}\to a\) and we conclude that \((e,\sqrt{a})\in\theta\) or \(\sqrt{a}\to e\), in which case applying \(L_{\sqrt{a}}\) yields that \(a\to\sqrt{a}\) and we conclude that \((\sqrt{a},a)\in\theta\). In either case if \(e/\theta\) contains \(a\) then \(e/\theta\) contains \(\sqrt{a}\). It follows that \(e/\theta\) is closed under taking square roots. Since \(G\) is finite we have that \(e/\theta\) is also closed under squaring. We have that \(e/\theta\supset\{e,a,a^2\}\).

We write \(a\theta\coloneqq L_a(\theta)\) for \(a\in G\) and \(\theta\in\con(\mathbf{A})\). Suppose inductively that \(e/\theta\supset\{a^{s-2},a^{s-1},a^s\}\). We show that \(a^{s+1}\in e/\theta\). Since \(e/\theta\) is an equivalence class of the principal congruence generated by \((e,a)\) applying \(L_a\) shows that \(a/(a\theta)\supset\{a^{s-1},a^s,a^{s+1}\}\) is an equivalence class of the principal congruence generated by \((a,a^2)\). We know that \((a,a^2)\in\theta\) so \(a\theta\le\theta\). Since \(a^{s-1},a^s\in e/\theta\) it follows that \(e/\theta\supset\{a^{s-1},a^s,a^{s+1}\}\) so \(a^{s+1}\in e/\theta\), as desired. We find that \(e/\theta\) contains \(\sg^{\mathbf{G}}(\{a\})\).

We now know that for any congruence \(\theta\in\con(\mathbf{A})\) we have that \(e/\theta\) is a union of cyclic subgroups of \(\mathbf{G}\). Suppose towards a contradiction that \(a,b\in e/\theta\) and \(ab\notin e/\theta\). Note that \(\theta\ge\cg(\{(e,a),(e,b^{-1})\})\). Observe that
	\begin{align*}
		\cg(\{(e,a),(e,b^{-1})\}) &= b^{-1}\cg(\{(b,ba),(b,e)\})\\
		&\ge b^{-1}\cg(\{(e,ba)\})\\
		&\ge b^{-1}\cg(\{(e,baba)\})\\
		&\ge \cg(\{(b^{-1},aba)\})
	\end{align*}
so we have that \(e/\theta\) contains \(aba\).

In order for \(ab\notin e/\theta\) we must have that \(ab\) either dominates everything in \(e/\theta\) or \(ab\) is dominated by everything in \(e/\theta\). In the former case we have that \(ab\to aba\). Applying \(L_{(ab)^{-1}}\) we find that \(e\to a\). Since \(ab\) dominates \(e/\theta\) we have that \((ab)^{-1}\) is dominated by \(e/\theta\). It follows that \(b^{-1}\to b^{-1}a^{-1}\) so \(e\to a^{-1}\). By definition of a regular \(\rps\) magma it is impossible to have both \(e\to a\) and \(e\to a^{-1}\) so we have arrived at a contradiction. The same argument with dominance relations reversed yields a contradiction in the case that \(ab\) is dominated by everything in \(e/\theta\). This establishes that \(e/\theta\) is a subset of \(G\) containing \(e\) which is closed under taking inverses and products. That is, \(e/\theta\) is a subgroup of \(\mathbf{G}\).

Given any \(\theta\in\con(\mathbf{A})\) and some \(a\in G\) we have that
	\[
		a/\theta=a(e/(a^{-1}\theta))=aH
	\]
for some \(\mathbf{H}\le\mathbf{G}\).
\end{proof}

We characterize those subgroups \(\mathbf{H}\le\mathbf{G}\) for which \(a/\theta=aH\) for some \(a\in G\).

\begin{definition}[\(\lambda\)-convex subgroup]
Given a group \(\mathbf{G}\), an \(n\)-sign function \(\lambda\in\sgn_n(\mathbf{G})\), and a subgroup \(\mathbf{H}\le\mathbf{G}\) we say that \(\mathbf{H}\) is \emph{\(\lambda\)-convex} when there exists some \(a\in G\) such that \(a/\theta=aH\) for some \(\theta\in\con(\mathbf{G}_n(\lambda))\).
\end{definition}

Trivially we have that the whole group \(\mathbf{G}\) and the trivial subgroup with universe \(\{e\}\) are both \(\lambda\)-convex for every \(\lambda\).

\begin{proposition}
Let \(\mathbf{G}\) be a finite group of order \(m\) and let \(n<\varpi(m)\). Take \(\lambda\in\sgn_n(\mathbf{G})\) and \(\mathbf{H}\le\mathbf{G}\). The following are equivalent:
	\begin{enumerate}
		\item The subgroup \(\mathbf{H}\) is \(\lambda\)-convex.
		\item There exists a congruence \(\psi\in\con(\mathbf{G}_n(\lambda))\) such that \(e/\psi=H\).
		\item Given \(1\le k\le n-1\) and \(b_1,\dots,b_k\notin H\) either \(e\to\{b_1h_1,\dots,b_kh_k\}\) for every choice of \(h_1,\dots,h_k\in H\) or \(\{b_1h_1,\dots,b_kh_k\}\to e\) for every choice of \(h_1,\dots,h_k\in H\).
	\end{enumerate}
\end{proposition}

\begin{proof}
To see that (1) implies (2) suppose that \(\mathbf{H}\) is \(\lambda\)-convex. We have that there exists some \(a\in G\) such that \(a/\theta=aH\) for some \(\theta\in\con(\mathbf{G}_n(\lambda))\). Applying \(L_{a^{-1}}\) we see that \(e/(a^{-1}\theta)=H\) so we can take \(\psi=a^{-1}\theta\).

To see that (2) implies (3) note that if \(e/\psi=H\) then we can define a congruence \(\theta_e\) where the only nontrivial equivalence class is \(H\). It follows that \(\theta\coloneqq\bigcup_{a\in G}a\theta_e\) is a congruence of \(\mathbf{G}_n(\lambda)\). Condition (3) is implied by \(\theta\) having the substitution property.

To see that (3) implies (1) note that if (3) holds then the equivalence relation \(\theta\) defined previously is a congruence and hence \(\mathbf{H}\) is \(\lambda\)-convex.
\end{proof}

The congruence lattice of \(\mathbf{G}_n(\lambda)\) is thus determined by those subgroups of \(\mathbf{G}\) which are \(\lambda\)-convex. For a fixed \(\lambda\) such subgroups must form a chain.

\begin{theorem}
Suppose that \(\mathbf{H},\mathbf{K}\le\mathbf{G}\) are both \(\lambda\)-convex. We have that \(\mathbf{H}\le\mathbf{K}\) or \(\mathbf{K}\le\mathbf{H}\).
\end{theorem}

\begin{proof}
Suppose that \(\mathbf{H}\) and \(\mathbf{K}\) are incomparable. That is, let \(h\in H\setminus K\) and let \(k\in K\setminus H\). Without loss of generality take \(h\to e\) and \(k\to e\). We have that \(hk^{-1}\to e\) since everything in the coset \(hK\) must dominate everything in \(K\) in order for \(K\) to be \(\lambda\)-convex. It follows that \(e\to kh^{-1}\) so everything in \(H\) dominates everything in \(kH\). This implies that \(e\to k\), a contradiction.
\end{proof}

As an immediate consequence we have the following proposition.

\begin{proposition}
If \(\mathbf{H}\) is a \(\lambda\)-convex subgroup of \(\mathbf{G}\) then \(\mathbf{H}\) is \(\aut_\lambda(\mathbf{G})\)-characteristic. That is, for any \(\varphi\in\aut_\lambda(\mathbf{G})\) we have that \(\varphi(\mathbf{H})=\mathbf{H}\). In particular, if \(\lambda\) is correlated then any \(\lambda\)-convex subgroup is normal.
\end{proposition}

\begin{proof}
Since \(\varphi\) is a \(\lambda\)-automorphism we have that \(\varphi(\mathbf{H})\) is \(\lambda\)-convex if and only if \(\mathbf{H}\) is \(\lambda\)-convex. By the previous theorem we have that \(\mathbf{H}\le\varphi(\mathbf{H})\) or \(\varphi(\mathbf{H})\le\mathbf{H}\). Since \(\lvert\mathbf{H}\rvert=\lvert\varphi(\mathbf{H})\rvert\) we find that \(\mathbf{H}=\varphi(\mathbf{H})\).
\end{proof}

We can now give the structure of the congruence lattice of \(\mathbf{G}_n(\lambda)\).

\begin{definition}[\(\lambda\)-coset poset]
Given \(\lambda\in\sgn_n(\mathbf{G})\) set
	\[
		P_\lambda\coloneqq\{\,aH\mid a\in G\text{ and }\mathbf{H}\text{ is }\lambda\text{-convex}\,\}
	\]
and define the \emph{\(\lambda\)-coset poset} to be \(\mathbf{P}_\lambda\coloneqq(P_\lambda,\subset)\).
\end{definition}

Dilworth showed that the maximal antichains of a finite poset form a distributive lattice. We follow Freese's treatment of this\cite{freese}. Given a finite poset \(\mathbf{P}\coloneqq(P,\le)\) let \(\mathbf{L}(\mathbf{P})\) be the lattice whose elements are maximal antichains in \(\mathbf{P}\) where if \(U,V\in L(\mathbf{P})\) then we say that \(U\le V\) in \(\mathbf{L}(\mathbf{P})\) when for every \(u\in U\) there exists some \(v\in V\) such that \(u\le v\) in \(\mathbf{P}\).

\begin{theorem}
We have that \(\conl(\mathbf{G}_n(\lambda))\cong\mathbf{L}(\mathbf{P}_\lambda)\).
\end{theorem}

\begin{proof}
Define \(h\colon\con(\mathbf{G}_n(\lambda))\to L(\mathbf{P}_\lambda)\) by \(h(\theta)\coloneqq\{\,a/\theta\mid a\in G\,\}\). By our previous work we have that \(h(\theta)\) is a subset of \(P_\lambda\). Note that since \(\theta\) is an equivalence relation we have that \(h(\theta)\) is an antichain and since every member of \(G\) must belong to some equivalence class under \(\theta\) we have that \(h(\theta)\) is a maximal antichain.

We have that \(h\) is a well-defined map. It remains to show that \(h\) is an isotone bijection with isotone inverse. Certainly \(h\) is injective, for if \(h(\theta)=h(\psi)\) then \(\theta\) and \(\psi\) determine the same partition of \(G\), which implies that \(\theta=\psi\). To see that \(h\) is surjective consider some maximal antichain \(U\) in \(\mathbf{P}_\lambda\). For each element \(aH\in U\) define \(\theta_{aH}\) to be the equivalence relation on \(G\) whose only nontrivial equivalence class is \(aH\). We have that \(\theta_{aH}\) is a congruence and thus so is \(\theta\coloneqq\bigcup_{aH\in U}\theta_{aH}\). For this choice of \(\theta\) we have that \(h(\theta)=U\) so \(h\) is surjective.

To see that \(h\) and \(h^{-1}\) are isotone note that one equivalence relation contains another precisely when the corresponding partition of one contains the partition of the other. This is equivalent to the given order on the antichains of \(\mathbf{P}_\lambda\).
\end{proof}

Since every lattice of maximal antichains is distributive we have that the finite regular \(\rps\) magmas are all congruence-distributive.

Our analysis also yields a family of simple magmas.

\begin{theorem}
Suppose that \(\mathbf{G}=\Z_{p^k}\) for a prime \(p\) and \(n<p\). There exists a \(\lambda\in\sgn_n(\mathbf{G})\) for which \(\mathbf{G}_n(\lambda)\) is simple.
\end{theorem}

\begin{proof}
Order the nontrivial subgroups of \(\mathbf{G}\) as \(\mathbf{H}_1\le\cdots\le\mathbf{H}_k=\mathbf{G}\). For each \(1\le i\le k-1\) choose a coset \(a+H_i\) of \(H_i\) other than \(H_i\) itself which lies in \(H_{i+1}\). Choose another element \(b\in a+H_i\) with \(b\neq a\). Set \(\lambda(\{a,-a\})\coloneqq a\) and \(\lambda(\{b,-b\})\coloneqq-b\). We have that \(\mathbf{H}_i\) is not \(\lambda\)-convex for \(1\le i\le k-1\). It follows that \(\mathbf{G}_n(\lambda)\) has no nontrivial proper \(\lambda\)-convex subgroups for this choice of \(\lambda\) so \(\mathbf{G}_n(\lambda)\) is simple.
\end{proof}

\printbibliography

\end{document}